\documentclass{article}

\usepackage[margin=3.4cm]{geometry}

\usepackage[utf8]{inputenc}

\usepackage[dvipsnames]{xcolor}

\usepackage{amsmath,amssymb,mathtools,amsfonts,amsthm,mathrsfs,physics}
\usepackage{mathtools}
\usepackage{thmtools}
\usepackage{yhmath}
\usepackage{calc}
\usepackage{tikz}
\usepackage{tikz-cd}
\usepackage{pgfplots}
\pgfplotsset{compat=1.18}
\usetikzlibrary{graphs,arrows,decorations.pathmorphing,decorations.markings,fit,positioning,hobby,arrows.meta}

\usepackage[hidelinks]{hyperref} %with hidden color
\usepackage[capitalize, nameinlink, noabbrev]{cleveref}

\usepackage{enumitem}
\usepackage{rotating}

\usepackage{booktabs}

\usepackage{graphicx}
\usepackage{subcaption}
\usepackage{float}

\usepackage{diagbox}
\usepackage{multirow}

\usepackage{array}

\usepackage[export]{adjustbox}

\theoremstyle{plain}
\newtheorem{theorem}{Theorem}[section]

\newtheorem{proposition}[theorem]{Proposition}

\newtheorem{corollary}[theorem]{Corollary}
\newtheorem{lemma}[theorem]{Lemma}

\theoremstyle{definition}
\newtheorem{definition}[theorem]{Definition}

\theoremstyle{remark}

\theoremstyle{remark}

\usepackage{adjustbox}

\newcommand{\R}{\mathbb{R}}

\newcommand{\define}[1]{{\bf \boldmath{#1}}}

\newcommand{\pder}[2][]{\frac{\partial#1}{\partial#2}}
\title{Minimal intersection radius for \texorpdfstring{$n$}{n} growing, non-homogeneous ellipsoids in \texorpdfstring{$\mathbb{R}^d$}{R^d}}

\author{Barbara Giunti\footnote{SUNY at Albany} \footnote{Corresponding author, \texttt{bgiunti@albany.edu}} \and Sean Hill$^*$ \and Felix X.-F. Ye$^*$}

\date{}
\begin{document}

\maketitle

\begin{abstract}
In this paper, we compute the Minimal Intersection Radius (MIR) of growing, non-homogeneous ellipsoids in arbitrary ambient dimension. 
We provide a geometric method to find the MIR using techniques from convex optimization, a secondary method using second-order cone programs, and show that the MIR can be phrased as an LP-type problem, where the computation from convex optimization acts as a certificate. 
We implement these methods and benchmark them using different convex solvers, and with or without the LP setting. 
We also provide a comparison with similar but different problems that appeared in the literature, and show that finding the MIR is, in general, not equivalent to finding the minimal enclosing ellipsoid.
\end{abstract}

\section{Introduction}\label{sec:intro}

In this paper, we find the minimal intersection radius (MIR) for a collection of growing non-homogeneous ellipsoids in arbitrary ambient dimension (see \cref{fig_growing_triple,fig_growing_mutual}). 
We assume that we are given a finite set of points with point-dependent ellipsoidal neighborhoods: the ``balls'' of radius $r$ around a given point are not circular shaped but rather ellipsoidal.
We emphasize these ellipsoids vary in orientation and growth speed from point to point.
\begin{figure}[htbp]
\centering
\newcommand{\triangleellipses}[1]{
\begin{tikzpicture}[baseline=(basepoint)]
\def\radius{#1}

\begin{scope}[blend mode=multiply]
\coordinate (1) at (0,0);
\coordinate (2) at (1,0);
\coordinate (3) at (0.5,0.88);

% A reference for baseline alignment
\coordinate (basepoint) at (1);

% ellipses aligned with opposite edges
\foreach \i/\a/\b in {1/2/3,2/1/3,3/1/2}{

    \pgfmathanglebetweenpoints
    {\pgfpointanchor{\a}{center}}
    {\pgfpointanchor{\b}{center}}
    \edef\angle{\pgfmathresult}

    \pgfmathsetmacro{\sizefactor}{1.05}
    \pgfmathsetmacro{\elong}{1.35}

    \pgfmathsetmacro{\aaxis}{\sizefactor*\elong*\radius}
    \pgfmathsetmacro{\baxis}{\sizefactor*(1/\elong)*\radius}

    \begin{scope}[shift={(\i)}, rotate=\angle]
        \fill[blue!35] (0,0) ellipse ({\aaxis} and {\baxis});
        %\draw[blue] (0,0) ellipse ({\aaxis} and {\baxis});
    \end{scope}
}

% Draw points last
\foreach \i in {1,...,3}{
    \fill (\i) circle (0.06);
}

\end{scope}
\end{tikzpicture}
}

% -------- Sequence of radii --------
\begin{center}
\foreach \r [count=\i] in {0,0.32,0.4,0.52,0.645,0.75}{
    \triangleellipses{\r}
    \ifnum\i<3
    \pgfmathsetmacro{\gap}{-0.1 + \r} 
    \hspace{\gap cm}
    \fi
    \ifnum\i>1
        \ifnum\i<6
            \pgfmathsetmacro{\gap}{-1 + \r} 
            \hspace{\gap cm}
        \fi
    \fi
}
\end{center}
\caption{A set of three points with growing ellipsoids around them (given by the Mahalanobis distance around the points), where the mutual intersection of the ellipsoids happens much earlier than their triple intersection.}
\label{fig_growing_triple}
\end{figure}
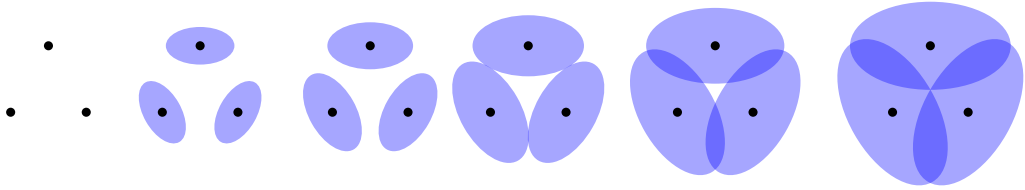
Thus, for each input point, its orientation matrix encodes its radius-$r$ ellipsoidal neighborhood.
Given a subset $S$ of the input points, the question is which is the smallest $r$ such that there is a common point in the radius-$r$ neighborhood of every element of $S$.
%, i.e., the intersection of the ellipsoids is non-empty. 
Since each point has an orientation matrix associated with it, and together they geometrically determine an ellipsoid, the question can be formulated as finding the minimal intersection radius of a collection of growing, non-homogeneous ellipsoids. 
The orientation matrices are assumed to be given. 
Once the MIR $r^*$ is known, one can normalize the problem by replacing every orientation matrix $A_i$ with $A_i/(r^*)^2$, so that the corresponding radius-$1$ ellipsoids intersect.
This is only a reparametrization, however, because the unknown scalar $r^*$ is precisely what we seek to compute.

The motivation for efficiently finding the MIR of growing ellipsoids comes from persistence theory, namely from the need for an efficient method to compute ellipsoidal \v Cech complexes, i.e., \v Cech complexes built from point-dependent ellipsoidal neighborhoods.
In turn, ellipsoidal \v Cech complexes are motivated by the study of manifold coverage for manifold learning. 
We will not dive into this topic in this work, but we briefly describe how we plan to use ellipsoidal \v Cech complexes for studying manifold coverage to motivate our study.
One may need to sample a manifold with a collection $C$ of points such that all points on the manifold are at a distance at most $\varepsilon$ from a point in $C$. 
To spot if and where this condition fails, one can compute the persistent homology of the \v Cech filtration of the sample: every point is a vertex, an edge is inserted when the two ellipsoids centered in its vertices have grown enough to touch, and a triangle enters when there is a triple intersection. 
When three edges are included but the corresponding triangle is not (as it would happen in the fourth and fifth step of \cref{fig_growing_triple}), there is a short bar in the persistent $1$-homology, i.e., a short-lived loop. 
One then simply goes to where in the sample the loop appeared, and adds a point in the middle of it, thus fixing the coverage. 
Or, having computed the MIR, one can adjust the value of $\varepsilon$ such that it is larger than most MIRs.
Beyond this motivation, we believe that finding the MIR for growing, non-homogeneous ellipsoids is interesting on its own from a purely computational geometry perspective. 

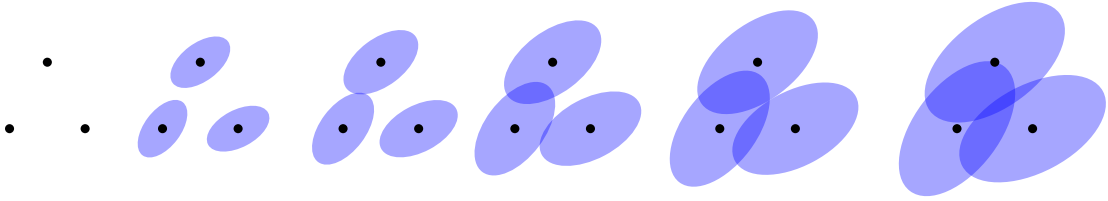
\begin{figure}
\centering
\newcommand{\triangleellipses}[1]{
\begin{tikzpicture}[baseline=(basepoint)]
\def\radius{#1}

\begin{scope}[blend mode=multiply]
\coordinate (1) at (0,0);
\coordinate (2) at (1,0);
\coordinate (3) at (0.5,0.88);

% A reference for baseline alignment
\coordinate (basepoint) at (1);

% Base orientation for ellipses
\pgfmathanglebetweenpoints
{\pgfpointanchor{1}{center}}
{\pgfpointanchor{2}{center}}
\edef\bottomangle{\pgfmathresult}

% same base orientation, slightly perturbed
\foreach \i/\delta in {1/13,2/-12,3/-3}{

    \pgfmathsetmacro{\sizefactor}{1 + 0.12*sin(10*\i)}
    \pgfmathsetmacro{\elong}{1.35}

    \pgfmathsetmacro{\aaxis}{\sizefactor*\elong*\radius}
    \pgfmathsetmacro{\baxis}{\sizefactor*(1/\elong)*\radius}

    \begin{scope}[shift={(\i)}, rotate=\bottomangle+40+\delta]

        \fill[blue!35] (0,0) ellipse ({\aaxis} and {\baxis});
        %\draw[blue] (0,0) ellipse ({\aaxis} and {\baxis});
    \end{scope}
}

% Draw points last
\foreach \i in {1,...,3}{
    \fill (\i) circle (0.06);
}
\end{scope}
\end{tikzpicture}
}

% -------- Sequence of radii --------

\begin{center}
\foreach \r [count=\i] in {0,0.32,0.4,0.52,0.645,0.75}{
    \triangleellipses{\r}
    \ifnum\i<3
    \pgfmathsetmacro{\gap}{-0.1 + \r} 
    \hspace{\gap cm}
    \fi
    \ifnum\i>1
        \ifnum\i<6
            \pgfmathsetmacro{\gap}{-1 + \r} 
            \hspace{\gap cm}
        \fi
    \fi
}
\end{center}
\caption{A set of three points with growing ellipsoids around them (given by the Mahalanobis distance around the points), where the mutual intersection of the ellipsoids happens at the same time as their triple intersection.}
\label{fig_growing_mutual}
\end{figure}

While several works have already touched on the topic of the intersection of ellipsoids, to the best of our knowledge, none dealt with the exact setting we require (see \cref{ssec:related_work} for a precise comparison). 
In detail, our setting is more general than the usual MIR problem for balls or homogeneous ellipsoids, as we do not assume that all ellipsoids in the collection have the same orientation matrix (see \cref{fig_growing_triple,fig_growing_mutual}). 
The lack of this assumption is an obstruction to the usual equivalence between finding the MIR of a collection of balls and finding the minimal enclosing ball (MEB) of the set of the centers of the balls (see \cref{ssec:diff_euclidean}). 

We find the MIR first by constructing a certificate using techniques from convex optimization. 
This certificate will take a set of centers and orientation matrices, compute a new ellipsoid that is guaranteed to cover the intersection, and then find the minimal value for which this ellipsoid is non-empty. 
This value will be the sought MIR (see \cref{ssec_convex}). 
Our certificate providing the radius is a generalization of the certificate for the intersection of two ellipsoids in \cite{robust2ellipsoids} (see \cref{ssec:related_work} for more details).
Then we will show how, with this certificate, we can phrase the problem as an LP-type problem (see \cref{ssec:lp}). 
Additionally, we express the MIR also as a second-order cone program (SOCP) in \cref{ssec_socp}. 
While this formulation is considerably shorter than the certificate we describe in \cref{ssec_convex}, it provides no explicit objective function nor clear geometric interpretation.
We ultimately plan to apply our method for \v Cech complexes for manifold learning, where the ambient dimension is very high and the number of ellipsoids to intersect is much lower, but our method still applies in lower ambient dimensions and for a large number of ellipsoids.

After having set the theoretical foundation of our method, we implemented it in \texttt{c++} using several different algorithms.
These algorithms are either direct solves (SLSQP \cite{NLopt,SLSQP}, the standard projected gradient descent (PGD), Cauchy \cite{chok2025convexoptimizationprobabilitysimplex}, and SOCP \cite[p. 156]{Boyd_Vandenberghe_2004} and \cite{alglib_library}), or LP-type algorithms (Clarkson \cite{clarkson} and Seidel \cite{seidel}).
We then tested the implementation on randomly generated ellipsoids, taking, for each computation, the mean over 30 tests. 
Each computation takes a fraction of a second. 
However, the long-term goal is to use the MIR computation as a subroutine of a larger algorithm, which would possibly be called hundreds, if not thousands, of times. 
Therefore, it was important to obtain methods that kept the computation in single-digit milliseconds in high ambient dimensions and for a large number of ellipsoids. 
%We showed that the direct solver SLSQP and the LP-Clarkson implementations kept these limits up to ambient dimension $50$ for $40$ ellipsoids (the maximum we computed). 
%In general, our experiments show that the direct solver SLSQP performed best.

The code, complete with the implementation and the experiment setup, is available at \cite{github_repo}. 

\subsection{Related work}\label{ssec:related_work}

The literature around ellipsoids contains several problems that are close to ours, but differ in the homogeneity assumptions, the number of ellipsoids, or the optimization objective. 
Intersections of ellipsoids and quadrics have been studied extensively, especially for concentric, coaxial, or otherwise structured families in $\R^d$ \cite{Medrano2023}; related algorithmic work also develops basic geometric operations for ellipsoids, including separation and enclosing routines \cite{pope:hal-04946526}. 
Other optimization work studies a fixed intersection of several ellipsoids, such as the problem of projecting a point onto that intersection \cite{LinHan2004}. 
Classical LP-type algorithms for smallest enclosing balls and ellipsoids \cite{ fischer_gaertner_kutz, GaertnerSchonherr,smallest_enclosing_welzl} also provide part of the algorithmic background for our later LP-type formulation. These works are adjacent to our setting, but they do not address the minimal radius at which a finite family of growing ellipsoids with different orientation matrices first acquires a common intersection.

For non-homogeneous ellipsoids, the exact pairwise literature we are aware of includes overlap tests for two hyperellipsoids \cite{robust2ellipsoids,6954724}, an algebraic separation criterion for two ellipsoids in $\R^3$ \cite{WANG2001531}, and a topology-determination strategy for computing the intersection curve of two ellipsoids in $\R^3$ \cite{algebraic_2ellipsoids}. We generalize to arbitrary $n$ some of the certificate ideas used in \cite{robust2ellipsoids} for $n=2$. However, those works test whether a fixed pair of ellipsoids overlaps, or compute the geometry of a fixed pairwise intersection; they do not compute the minimal intersection radius for a family of arbitrary size. 
The closest source to our convex-optimization viewpoint is \cite[\S 5.8.3, ``Intersection of ellipsoids'' example]{Boyd_Vandenberghe_2004}, where the authors construct an ellipsoid $E_\lambda$ covering the intersection of non-homogeneous ellipsoids via Lagrange duality. Their treatment gives a feasibility certificate for a fixed radius, whereas our goal is to optimize over the radius itself and then use the resulting certificate inside an LP-type formulation. 
Since our motivating application is the computation of ellipsoidal \v Cech complexes, a related but different construction given the Rips-type ellipsoid complex of \cite{PHellipsoids} cannot be adopted because it avoids the full common-intersection test that motivates the MIR computation.

Finding the MIR is also distinct from finding the Minimum Volume Intersection Covering Ellipsoid (MVICE), a problem whose natural primal is nonconvex but for which standard SDP, S-procedure, and bounding-ellipsoid relaxations are convex and equivalent \cite{relaxMVICE}. 
The MVICE problem takes fixed input ellipsoids and asks for a minimum-volume ellipsoid covering their intersection \cite{relaxMVICE}. 
In our problem, deciding whether the intersection is non-empty at a given radius and finding the first radius at which this happens are precisely the central tasks. 
Our covering ellipsoid $E_{\lambda,r}$ (see Eq.\eqref{Elambda}) is therefore a certificate-dependent cover of the intersection, not a volume-optimal cover. 
When the intersection is non-empty, $E_{\lambda,r}$ is feasible for the corresponding covering problem, but no set containment relation with the MVICE follows in general. 
A still different problem asks for the smallest ball intersecting every member of a family of compact convex bodies of fixed size \cite{zheng2025approximation}.

Another related line of work focuses on the Minimal Enclosing Ellipsoid (MEE), a classical problem studied for decades; see, for example, \cite{Behrend1938,Kumar2005MinimumVolume,Post_spanning_ellipsoid} and references therein. We use the MEE only as a point of comparison. As discussed in \cref{ssec:diff_euclidean}, the Euclidean equivalence between the intersection radius of growing balls and the minimal enclosing ball of their centers does not extend to non-homogeneous ellipsoids. Moreover, for finite point sets, the MEE formulation is sensitive to degeneracy and to whether one requires a full-dimensional enclosing ellipsoid; even when the MEE is well-posed, it is not equivalent to the MIR considered here.

To help summarize the various problems, we collect the main acronyms, some of which we redefine from the literature, since it can be challenging to differentiate between them:

\begin{itemize}
\item[] MIR$(M)$: Minimal intersection radius of a finite family $M$ of growing sets in $\R^d$.
\item[] MEB$(M)$ and MEE$(M)$: Minimal enclosing ball and ellipsoid of an object $M$ in $\R^d$, where ``minimal'' refers to the volume. 
In the case of MEE, whether $M$ is a (convex) region or a collection of points has a big impact on the well-definedness of the problem.
\item[] MVICE: Minimum Volume Intersection Covering Ellipsoid, i.e., the ellipsoid with the smallest volume covering the intersection of (convex) bodies (often, other ellipsoids).
\end{itemize}

\section{Minimal Intersection Radius}\label{sec:mir}

Throughout the work, the symbol  $S_{++}^d$ denotes the set of symmetric positive definite (PD) ($d\times d$)-matrices with entries in $\mathbb{R}$.
For the reader's convenience, the main symbols used in this section are collected in \cref{tab:notation} of \cref{app:notation}; each is defined in detail where it first appears.

\begin{definition}\label{def_ellipse_fixed_r}
An \define{ellipsoid} (or \define{ellipsoidal ball}) in $\R^d$ is the set
\[
E(r,\tilde c, A) \coloneqq \{x\in \R^d\colon (x-\tilde c)^\top A (x-\tilde c) \leqslant r^2\}\, ,
\]
where $r\geqslant 0$ and it is called the \define{radius}, $A\in S_{++}^d$ is a symmetric positive definite (PD) $d\times d$ matrix called the \define{orientation matrix}\footnote{$A$ is also called \define{precision matrix} or \define{inverse covariance} in the literature.}, and $\tilde c\in\R^d$ is the \define{center}. 
\end{definition}

For a geometric interpretation of ellipsoids, one can define $\norm{z}_A \coloneqq \sqrt{z^\top A z}$, i.e., as the Mahalanobis norm in $\mathbb{R}^d$. 
Then $E(r, \tilde c, A)$ is a closed ball of radius $r$ under the Mahalanobis distance with orientation matrix $A$.

In the above definition, the radius, orientation matrix, and center of an ellipsoid are all fixed. 
However, the geometric interpretation motivates the study of ellipsoids with varying radius: one can consider all points at a certain distance from a center, and then increase or decrease such distance, including or excluding points. 

\begin{definition}\label{def_ellipse_grow_r}
A \define{growing ellipsoid} in $\R^d$ is an ellipsoid $E(r,\tilde c, A)$
where $r$ varies in $[0,\infty)$.  
We denote a growing ellipsoid with $E_r(\tilde c, A)$. 
\end{definition}

The question we answer in this paper can be formulated as follows: 
\begin{itemize}
    \item[\hypertarget{prob}{\textbf{Q:}}] Given $n$ centers $\tilde c_1, \dotsc, \tilde c_n \in\R^d$ and orientation matrices $A_1,\dotsc, A_n \in S_{++}^d$, find
    $r^* = \inf \{r \geqslant 0 \colon \displaystyle\bigcap_{i=1}^n E_r(\tilde c_i, A_i) \neq \emptyset\}\,$,
    i.e., the smallest radius such that the intersection of $n$ non-homogeneous, growing ellipsoids in $\mathbb{R}^d$ is non-empty. 
\end{itemize}

We note that, for any $r>0$, we have
$E(r,\tilde c, A) = E(1,\tilde c, A/r^2)$, thus one could reformulate \hyperlink{prob}{\textbf{Q}} as finding the minimal rescaling factor of the orientation matrices such that the ellipsoids at radius $1$ have a common, non-empty intersection. 
However, as we found the original formulation easier to treat and efficient to implement, we did not pursue this alternative angle further.
\medskip

Before going into the details of the solution of \hyperlink{prob}{\textbf{Q}}, we recall two facts about ellipsoids that will be useful in our computation. 
To keep the work self-contained, we used the notation for growing ellipsoids, which we will maintain throughout. 
Moreover, we included the brief proofs of \cref{ellipse_as_quadratic_preimage,gradient_hessian_and_min_of_f} even if they are classical results.

\begin{lemma}\label{ellipse_as_quadratic_preimage}
	Let $A\in S_{++}^d$, $b\in\R^d$ and $c\in\R$. For each $x\in\R^d$, define $f(x)=x^\top A x + 2b^\top x + c$.
	Then the closed ellipsoidal ball centered at $\tilde c$ with orientation matrix $A$ can be written as the following preimage under $f$:
	$E_r(\tilde c,A)=f^{-1}((-\infty, 0]) = \{x\in\R^d :f(x)\leqslant 0\}$,
	provided $b = -A \tilde c \text{ and } c = \tilde c^\top A \tilde c-r^2$.
\end{lemma}

\begin{proof}
	Expanding $(x-\tilde c)^\top A (x-\tilde c)-r^2$ (using $A^\top=A$) shows it equals $x^\top A x+2b^\top x+c$ precisely when $b=-A\tilde c$ and $c=\tilde c^\top A\tilde c-r^2$.
\end{proof}

Alternatively, given $(A,b,c)$ we can solve for $\tilde c$ and $r$ via $\tilde c = -A^{-1} b$ and $r^2 = \tilde c^\top A \tilde c - c$.

\begin{lemma}\label{gradient_hessian_and_min_of_f}
	The gradient of $f$ is $\nabla f(x) = 2A x+2b$, the Hessian is $\nabla^2 f(x) = 2A \in S_{++}^d$, and the minimizer occurs at $x^*=-A^{-1} b$ with minimum $f(x^*) = c-b^\top A^{-1} b$.
\end{lemma}

\begin{proof}
	These follow from direct vector calculus: $\nabla f(x)=2Ax+2b=0$ gives the unique minimizer $x^*=-A^{-1}b$, and substituting yields $f(x^*)=c-b^\top A^{-1}b$.
\end{proof}

\subsection{MIR as a convex optimization problem}\label{ssec_convex}

We now expand on the discussion in \cite[p. 258--262]{Boyd_Vandenberghe_2004} to reformulate \hyperlink{prob}{\textbf{Q}} as a feasibility problem and solve it as a convex optimization. 
\medskip 

For each $r\geqslant 0$, write $\gamma_i(r) = \tilde c_i^\top A_i \tilde c_i - r^2$, $b_i = -A_i\tilde c_i$, and $f_i(x; r)= x^\top A_i x + 2b_i^\top x + \gamma_i(r)$.
Then, by \cref{ellipse_as_quadratic_preimage}, for every $i=1,2,\dotsc, n$ and for each $r\geqslant 0$ fixed, we can rewrite a growing ellipsoid as 
$E_r(\tilde c_i, A_i) = f_i(\cdot;r)^{-1} ((-\infty, 0])\,$.

Thus, for each $r\geqslant 0$, 
\[
\displaystyle\bigcap_{i=1}^n E_r(\tilde c_i, A_i) \neq \emptyset
\iff
\underbrace{f_i(x;r) = x^\top A_i x + 2b_i^\top x + \gamma_i(r) \leqslant 0 \ \ \ \forall \, 1\leqslant i\leqslant n}_{\hypertarget{syst}{(\star)}} \, .
\]
The right-hand side \hyperlink{syst}{$(\star)$} is a system of non-strict inequalities giving a feasibility problem, and we solve it using the method of Lagrange multipliers: 

\[
\mathcal{L}_r(x, \lambda) = \sum_{i=1}^n \lambda_i f_i(x; r) = x^\top A(\lambda) x + 2b(\lambda)^\top x + \gamma(\lambda;r)
\]
for any $x\in\R^d$ and $\lambda\in \R^n$, where $A(\lambda) = \sum_{i=1}^n \lambda_i A_i$, $b(\lambda) = \sum_{i=1}^n \lambda_i b_i$, and, since $\gamma_i(r)=\tilde c_i^\top A_i \tilde c_i - r^2$,
\[
\gamma(\lambda;r) = \sum_{i=1}^n \lambda_i \gamma_i(r) = \sum_{i=1}^n \lambda_i \tilde c_i^\top A_i \tilde c_i - r^2 \sum_{i=1}^n \lambda_i\, .
\]
Using Lemma \ref{gradient_hessian_and_min_of_f}, the dual function is, by definition,
\begin{align*}
	g(\lambda;r) &= \inf_{x\in\R^d} \left(x^\top A(\lambda) x + 2b(\lambda)^\top x + \gamma(\lambda;r)\right)\\
	& = \begin{cases*}
		\gamma(\lambda;r) - b(\lambda)^\top A(\lambda)^\dagger b(\lambda), & $A(\lambda) \succeq 0, \, b(\lambda) \in\mathcal{R}(A(\lambda))$\\
		-\infty & \text{ otherwise}
	\end{cases*} \, ,
\end{align*}
where $\mathcal{R}$ is the range of a linear transformation. 

For arbitrary $\lambda\in\R^n$, we are not guaranteed that $A(\lambda)$ is positive definite, so we must use the Moore-Penrose inverse $A(\lambda)^\dagger$ above. 
If $\lambda_i \geqslant 0$ and $\lambda \neq 0$ (as it will be the case in our setting), then $A(\lambda)=\sum_{i=1}^n \lambda_i A_i$ is a non-negative combination of the PD matrices $A_i$ with at least one positive weight, hence $A(\lambda) \in S_{++}^d$, and the dual function simplifies to
$g(\lambda;r) = \gamma(\lambda;r) - b(\lambda)^\top A(\lambda)^{-1} b(\lambda)\,$.

For easing the notation, we set
\begin{equation}\label{eq_def_g}
    g(\lambda;r)\coloneqq \gamma(\lambda;r) - b(\lambda)^\top A(\lambda)^{-1} b(\lambda) \, .
\end{equation}

We now derive a feasibility certificate for \hyperlink{syst}{$(\star)$} via Lagrangian strong duality applied to an epigraph reformulation.
For each fixed $r\geqslant 0$, define
\begin{equation}\label{eq_primal_epigraph}
    p(r) \coloneqq \inf_{\substack{x\in\R^d\\ t\in\R}} \{t \colon f_i(x; r) \leqslant t,\ i=1,2,\dotsc, n\}
    = \inf_{x\in\R^d} \max_{1\leqslant i\leqslant n} f_i(x; r) \, .
\end{equation}
The intersection $\bigcap_{i=1}^n E_r(\tilde c_i, A_i)$ is non-empty if and only if there exists $x\in\R^d$ with $\max_i f_i(x; r) \leqslant 0$, that is, if and only if $p(r)\leqslant 0$.

Let $\Delta^n \coloneqq \{\lambda\in\R^n : \lambda_i \geqslant 0 \text{ for every } i,\; \sum_{i=1}^n \lambda_i = 1\}$ denote the \define{probability simplex}, i.e., the standard $n$-simplex in $\R^n$ with vertices the standard basis vectors $e_i=(0,\dotsc,0,\underset{i}{1},0,\dotsc,0)$.

\begin{lemma}\label{strong_duality_for_p}
    For every $r\geqslant 0$, the infimum in~\eqref{eq_primal_epigraph} is attained and
    \begin{equation}\label{eq_strong_duality}
        p(r) = \max_{\lambda\in\Delta^n} g(\lambda; r) \, .
    \end{equation}
\end{lemma}

\begin{proof}
    The objective $t$ and the constraints $f_i(x; r) - t \leqslant 0$ are convex in $(x,t)$ (each $f_i(\cdot; r)$ is a convex quadratic since $A_i\in S_{++}^d$).
    Slater's condition holds: fixing any $x_0\in\R^d$ and choosing $t_0 > \max_i f_i(x_0; r)$ gives a strictly feasible pair.
    Moreover, $\max_i f_i(\cdot; r)$ is coercive (each $f_i(x; r)\to\infty$ as $\norm{x}\to\infty$ by positive-definiteness of $A_i$), so the infimum in~\eqref{eq_primal_epigraph} is attained.

    The Lagrangian, with multipliers $\lambda_i \geqslant 0$ for the constraints $f_i(x; r) - t \leqslant 0$, is
    \[
        L(x, t, \lambda)
        = t + \sum_{i=1}^n \lambda_i\bigl(f_i(x; r) - t\bigr)
        = \Bigl(1 - \textstyle\sum_{i=1}^n \lambda_i\Bigr)\, t + \mathcal{L}_r(x, \lambda) \, .
    \]
    The infimum over $t\in\R$ is $-\infty$ unless $\sum_{i=1}^n \lambda_i = 1$, in which case $L$ no longer depends on $t$ and reduces to $\mathcal{L}_r(x, \lambda)$.
    Minimizing then over $x\in\R^d$ gives, for $\lambda\in\Delta^n$,
    $\inf_x \mathcal{L}_r(x, \lambda) = g(\lambda; r)$
    by \cref{gradient_hessian_and_min_of_f} applied to the quadratic $\mathcal{L}_r(\,\cdot\,, \lambda)$ (whose Hessian $2A(\lambda)$ is positive definite since $\lambda\in\Delta^n$).
    Hence the dual function of the epigraph problem~\eqref{eq_primal_epigraph} equals $g(\,\cdot\,; r)$ on its effective domain $\Delta^n$ and $-\infty$ elsewhere, so the dual problem is $\sup_{\lambda\in\Delta^n} g(\lambda; r)$.
    Slater's condition implies strong duality, so $p(r) = \sup_{\lambda\in\Delta^n} g(\lambda; r)$; compactness of $\Delta^n$ together with continuity of $g(\,\cdot\,; r)$ on $\Delta^n$ (the matrix $A(\lambda)$ stays positive definite over $\Delta^n$, so $A(\lambda)^{-1}$ is continuous in $\lambda$) makes the supremum a maximum.
\end{proof}

Combining the feasibility equivalence $\bigcap_i E_r(\tilde c_i, A_i)\neq\emptyset \iff p(r)\leqslant 0$ with~\eqref{eq_strong_duality} yields the certificate of intersection
\begin{equation}\label{eq_cert_in_iff}
    \displaystyle\bigcap_{i=1}^n E_r(\tilde c_i, A_i)\neq\emptyset
    \ \iff\
    \max_{\lambda\in\Delta^n} g(\lambda; r) \leqslant 0 \, .
\end{equation}

For $\lambda\in\Delta^n$, the value $g(\lambda; r)$ has a geometric interpretation: define the ellipsoid
\begin{equation}\label{Elambda}
    E_{\lambda, r} \coloneqq \{x\in\R^d\colon  \mathcal{L}_r(x,\lambda) \leqslant 0\} \, ,
\end{equation}
which contains the intersection of the input ellipsoids $\bigcap_{i=1}^n E_r(\tilde c_i, A_i)$.
Indeed, for any $\lambda\in\Delta^n$, if $f_i(x; r) \leqslant 0$ for all $i=1,2,\dotsc,n$, then $\mathcal{L}_r(x, \lambda) = \sum_i \lambda_i f_i(x; r) \leqslant 0$, so $E_{\lambda, r}\supseteq \bigcap_{i=1}^n E_r(\tilde c_i, A_i)$.

Note that, in general, when the latter intersection contains more than a single point, $E_{\lambda, r}$ is \emph{not} the minimal ellipsoid containing said intersection.
However, this is not a requirement for the solution of \hyperlink{prob}{\textbf{Q}}: as we will see shortly, to solve \hyperlink{prob}{\textbf{Q}}, we simply need to identify when $E_{\lambda, r}$ contains exactly one point, in which case $\bigcap_{i=1}^n E_r(\tilde c_i, A_i)$ also contains exactly one point and we found the desired MIR.
\medskip

While the discussion so far has largely been along the lines of \cite{Boyd_Vandenberghe_2004}, what follows is our original contribution, generalizing to an arbitrary number $n$ of ellipsoids what was done in \cite{robust2ellipsoids} for $n=2$.

\begin{proposition}\label{covering_set_is_ellipse}
	Let $\lambda\succeq 0$ and $\lambda\neq 0$. The set $E_{\lambda, r}$ is a possibly empty or degenerate Mahalanobis closed ball with center $m(\lambda) = -A(\lambda)^{-1} b(\lambda)$, orientation matrix $A(\lambda)$, and formal squared radius
    $\epsilon_r(\lambda)^2 = m(\lambda)^\top A(\lambda)m(\lambda)-\gamma(\lambda;r)=-g(\lambda;r)\,$.
    In particular,
    $E_{\lambda, r} = \{y\in\R^d\colon \norm{y-m(\lambda)}_{A(\lambda)}^2 \leqslant \epsilon_r(\lambda)^2\}\,$.
    Further, for any $\lambda\succeq 0$ and $\lambda\neq0$, we have
 	\[
    \displaystyle\bigcap_{i=1}^n E_r(\tilde c_i, A_i) \subseteq E_{\lambda, r} \subseteq \displaystyle\bigcup _{i=1}^n E_r(\tilde c_i, A_i)\, .
    \]
\end{proposition}

\begin{proof}
	Since $\lambda\succeq 0$ and $\lambda\neq 0$, the matrix
	$A(\lambda)=\sum_i\lambda_iA_i$ is positive definite, and $m(\lambda)$ is well-defined.
	For readability, write $A=A(\lambda)$, $b=b(\lambda)$, $\gamma=\gamma(\lambda;r)$, and $m=m(\lambda)$.
	From $m=-A^{-1}b$, we have $b=-Am$. Completing the square gives, for any $y\in\R^d$,
	\begin{align*}
	\mathcal{L}_r(y,\lambda)
	=y^\top Ay+2b^\top y+\gamma=(y-m)^\top A(y-m)+\gamma-m^\top Am.
	\end{align*}
	Moreover, $m^\top Am=b^\top A^{-1}b$, hence $\gamma-m^\top Am=g(\lambda;r)$.
	Therefore
	\[
	E_{\lambda,r}
	=\left\{y\in\R^d\colon \norm{y-m(\lambda)}_{A(\lambda)}^2\leqslant -g(\lambda;r)\right\},
	\]
	which is the claimed description, with $\epsilon_r(\lambda)^2=-g(\lambda;r)$.

	For the first containment, let $y\in\cap_i E_r(\tilde c_i,A_i)$. Then $f_i(y;r)\leqslant 0$ for every $i$, and since $\lambda_i\geqslant 0$, we have
$\mathcal{L}_r(y,\lambda)=\sum_{i=1}^n\lambda_i f_i(y;r)\leqslant 0$.
	Thus $y\in E_{\lambda,r}$.
	For the second containment, let $y\in E_{\lambda,r}$. If $y\notin\cup_iE_r(\tilde c_i,A_i)$, then $f_i(y;r)>0$ for every $i$. Since $\lambda\succeq0$ and $\lambda\neq0$, at least one $\lambda_i$ is positive, and it follows that
$\mathcal{L}_r(y,\lambda)=\sum_{i=1}^n\lambda_i f_i(y;r)>0$,
	contradicting $y\in E_{\lambda,r}$. Hence $y\in\displaystyle\bigcup_iE_r(\tilde c_i,A_i)$.
\end{proof}

\begin{corollary}\label{classification_of_covering}
	Let $\lambda\succeq0$ and $\lambda\neq0$.
    Then the following hold:
	\begin{enumerate}
		\item If $g(\lambda;r)<0$ then $E_{\lambda, r}$ is a non-degenerate ellipsoid;
		\item If $g(\lambda;r)=0$ then $E_{\lambda, r}=\{m(\lambda)\}$;
		\item If $g(\lambda;r)>0$ then $E_{\lambda, r}=\emptyset$.
	\end{enumerate}
\end{corollary}

\begin{proof}
	By \cref{covering_set_is_ellipse},
	$E_{\lambda,r}=\left\{y\in\R^d\colon \norm{y-m(\lambda)}_{A(\lambda)}^2\leqslant -g(\lambda;r)\right\}$.
	Since $A(\lambda)$ is positive definite, the left-hand side of the inequality is non-negative and vanishes only at $y=m(\lambda)$.
	If $g(\lambda;r)<0$, then $-g(\lambda;r)>0$, so the set is a non-degenerate ellipsoid.
	If $g(\lambda;r)=0$, then the inequality forces $y=m(\lambda)$.
	If $g(\lambda;r)>0$, then the right-hand side is negative, so no point can satisfy the inequality.
\end{proof}

\cref{covering_set_is_ellipse} and \cref{classification_of_covering} are the generalizations of Propositions 1 and 2 in \cite{robust2ellipsoids}, respectively. 
We note that the maximizers of $g(\lambda;r)$ on $\Delta^n$ do not depend on $r$.
Indeed, we can rewrite $g$ as
$g(\lambda; r) = \sum_{i=1}^n \lambda_i \tilde c_i^\top A_i \tilde c_i - b(\lambda)^\top A(\lambda)^{-1} b(\lambda) -r^2 \,$,
where the $r$-dependent term is constant in $\lambda$.
Therefore, we can simply focus on the function
\[
h(\lambda)\coloneqq \sum_{i=1}^n \lambda_i \tilde c_i^\top A_i \tilde c_i - b(\lambda)^\top A(\lambda)^{-1} b(\lambda) = a^\top \lambda -m(\lambda)^\top A(\lambda) m(\lambda) \, ,
\]
where $a_i = \tilde c_i^\top A_i \tilde c_i=\norm{\tilde c_i}_{A_i}^2$ and we used $b(\lambda)^\top A(\lambda)^{-1} b(\lambda)=m(\lambda)^\top A(\lambda) m(\lambda)$, the center of $E_{\lambda, r}$ from \cref{covering_set_is_ellipse}.
The formulation of $h(\lambda)$ can be further modified as the next lemma shows, to obtain a form that will be useful in \cref{derivatives_of_objective_function,obj_func_convex}. 

\begin{lemma}\label{negative_obj_function_is_convex_combination_of_mahadistances}
	The function $h(\lambda)$ defined above can be written as
    $h(\lambda) = \sum_{i=1}^n \lambda_i \norm{\tilde c_i - m(\lambda)}_{A_i}^2$.
\end{lemma}

\begin{proof}
	Expanding $(\tilde c_i-m)^\top A_i (\tilde c_i - m)$, multiplying by $\lambda_i$, and summing gives
	\begin{align*}
		\sum_{i=1}^n \lambda_i (\tilde c_i-m)^\top A_i (\tilde c_i - m) &= a^\top \lambda -2 \sum_{i=1}^n \lambda_i \tilde c_i^\top A_i m(\lambda)+ m(\lambda)^\top A(\lambda) m(\lambda)\\
		& = a^\top \lambda -m(\lambda)^\top A(\lambda) m(\lambda) = h(\lambda),
	\end{align*}
	where the cross term simplifies via $\sum_{i=1}^n \lambda_i \tilde c_i^\top A_i m(\lambda)=m(\lambda)^\top A(\lambda)m(\lambda)$, which follows from the definitions $b_i=-A_i \tilde c_i$, $b(\lambda)=\sum_{i=1}^n \lambda_i b_i$, and $m(\lambda) = -A(\lambda)^{-1} b(\lambda)$.
\end{proof}

To find the maximum of $h(\lambda)$, which will give us the maximum of $g(\lambda;r)$ on $\Delta^n$, which, in turn, by \cref{classification_of_covering}, \eqref{eq_def_g}, and \eqref{eq_cert_in_iff}, will give us the sought MIR, we convert the question into a minimization problem:
\[
\min_{\lambda \in \Delta^n} F(\lambda) =
\min_{\lambda \in \Delta^n} \underbrace{m(\lambda)^\top A(\lambda) m(\lambda)-a^\top \lambda}_{-h(\lambda)} \, .
\]
We first collect some computations regarding the objective function $F=-h$.

\begin{lemma}\label{derivatives_of_objective_function}
	The gradient vector of $F$ is given by the entries $j=1,2,\dotsc, n$, 
	\[\pder[F]{\lambda_j} = -\norm{\tilde c_j - m(\lambda)}_{A_j}^2 \leqslant 0.\]
	The Hessian matrix of $F$ is
	\[\pder[^2 F]{\lambda_i \lambda_j} = 2 \langle w_i(\lambda), w_j(\lambda)\rangle_{A(\lambda)^{-1}},\]
	where $w_i(\lambda) = A_i(m(\lambda)-\tilde c_i)$, and $\langle x, y\rangle_A = x^\top A y$ for any $x,y$ vectors and $A$ positive definite matrix.
\end{lemma}

\begin{proof}
	Write $A=A(\lambda)$, $b=b(\lambda)$, $m=m(\lambda)$, and $a_j=\tilde c_j^\top A_j\tilde c_j$.
	Since $m=-A^{-1}b$, we have
	$F(\lambda)=m^\top Am-a^\top\lambda=b^\top A^{-1}b-a^\top\lambda$.
	Using $\partial A/\partial\lambda_j=A_j$, $\partial b/\partial\lambda_j=b_j=-A_j\tilde c_j$, and
$\frac{\partial A^{-1}}{\partial\lambda_j}=-A^{-1}A_jA^{-1}$,
	we obtain
	\begin{align*}
	\pder[F]{\lambda_j}
	&=2b_j^\top A^{-1}b-b^\top A^{-1}A_jA^{-1}b-a_j=2\tilde c_j^\top A_jm-m^\top A_jm-\tilde c_j^\top A_j\tilde c_j=-(\tilde c_j-m)^\top A_j(\tilde c_j-m).
	\end{align*}
	This gives the claimed gradient formula.
It remains to differentiate the gradient. From $m=-A^{-1}b$, we get
	$\pder[m]{\lambda_i}
	=A^{-1}A_iA^{-1}b-A^{-1}b_i
	=-A^{-1}A_i(m-\tilde c_i)
	=-A^{-1}w_i(\lambda)$.
	Therefore,
	\begin{align*}
	\pder[^2F]{\lambda_i\lambda_j}
	&=\pder{}{\lambda_i}\left[-(\tilde c_j-m)^\top A_j(\tilde c_j-m)\right]=2(\tilde c_j-m)^\top A_j\pder[m]{\lambda_i}\\
	&=2(m-\tilde c_j)^\top A_jA^{-1}w_i(\lambda)=2w_j(\lambda)^\top A(\lambda)^{-1}w_i(\lambda),
	\end{align*}
	which is the stated Hessian formula.
\end{proof}

The gradient formula admits a transparent geometric reading: each partial derivative $\pder[F]{\lambda_j}=-\norm{\tilde c_j-m(\lambda)}_{A_j}^2$ is the negative of a squared Mahalanobis radius, hence non-positive for every $j$. Thus $F$ is non-increasing along each coordinate direction, and its minimizer over $\Delta^n$ is pinned down by the equality constraint $\sum_{i}\lambda_i=1$ rather than by an interior stationary point of the unconstrained objective. This sign-definiteness underlies the shape of the hypersurface $F(\lambda)=-h(\lambda)$ visualized in \cref{fig:triple_none,fig:double}.

\begin{lemma}\label{obj_func_convex}
	The objective function $F\colon \Delta^n\to \R$ given by
$F(\lambda)=m(\lambda)^\top A(\lambda) m(\lambda)-a^\top \lambda$
	is convex on all of $\Delta^n$.
\end{lemma}

\begin{proof}
	We recall that a twice-differentiable function on a convex set is convex if its Hessian matrix is positive semi-definite on the interior of the convex set.	
	Since $A(\lambda)^{-1}$ is positive definite for each $\lambda\in \Delta^n$, we obtain for any vector $v=(v_1,\dotsc, v_n)^\top$, that, using Lemma \ref{derivatives_of_objective_function},
	\begin{align*}
		\sum_{i,j} v_i v_j \pder[^2 F]{\lambda_i \lambda_j} &= 2 \left\langle \sum_i v_i \vec{w}_i(\lambda), \sum_j v_j \vec{w}_j(\lambda)\right\rangle_{A(\lambda)^{-1}}
        =2\langle w(\lambda), w(\lambda) \rangle_{A(\lambda)^{-1}} \geqslant 0
	\end{align*}
	where $w=\displaystyle\sum_{k=1}^n v_k \vec{w}_k(\lambda)$. 
    This holds for all $\lambda \in \Delta^n$, hence $F$ has a PSD Hessian on the interior of the convex domain $\Delta^n$, hence $F$ is convex on all of $\Delta^n$.
\end{proof}

Because $F$ is convex and differentiable, the constraints $\lambda_i\geqslant 0$ and $1^\top\lambda=1$ are affine, and Slater's condition holds (the relative interior of $\Delta^n$ is non-empty), the KKT conditions are necessary and sufficient for optimality. Introduce the multiplier $\alpha\in\R$ for the equality constraint $1^\top\lambda=1$ and multipliers $\mu_1,\dotsc,\mu_n\geqslant 0$ for the inequalities $\lambda_i\geqslant 0$, giving the Lagrangian
$\mathcal{L}(\lambda, \alpha, \mu) = F(\lambda)+\alpha(1^\top \lambda-1) - \mu^\top \lambda\,$.
Stationarity $\nabla_\lambda\mathcal{L}=0$, evaluated with \cref{derivatives_of_objective_function}, gives for each $i=1,2,\dotsc,n$, $\norm{\tilde c_i - m(\lambda)}_{A_i}^2 = \alpha - \mu_i\,$.

Since $\mu_i\geqslant 0$ and the left-hand side is non-negative, $\alpha\geqslant 0$ and $\norm{\tilde c_i - m(\lambda)}_{A_i}^2\leqslant\alpha$ for every $i$; complementary slackness $\mu_i\lambda_i=0$ forces $\mu_i=0$, hence $\norm{\tilde c_i - m(\lambda)}_{A_i}^2=\alpha$, at each index with $\lambda_i>0$. As $\lambda\in\Delta^n$ has at least one positive component, it follows that
\[
\alpha = \max_{1\leqslant i\leqslant n} \norm{\tilde c_i -m(\lambda)}_{A_i}^2
\qquad\text{and}\qquad
m(\lambda) \in \displaystyle\bigcap_{i=1}^n E_{\sqrt{\alpha}}(\tilde c_i, A_i)\, .
\]

Thus, we have proven
\begin{lemma}\label{kkt_centroid}
	Let $\lambda^*$ be a minimizer of $F$ over $\Delta^n$ and $\alpha = \max_{1\leqslant i\leqslant n} \norm{\tilde c_i -m(\lambda^*)}_{A_i}^2$.
	Then
    $m(\lambda^*) \in \displaystyle\bigcap_{i=1}^n E_{\sqrt{\alpha}}(\tilde c_i, A_i)\,$.
	Moreover, $\norm{\tilde c_i - m(\lambda^*)}_{A_i}=\sqrt{\alpha}$ at every index $i$ with $\lambda_i^*>0$; in particular, if $\lambda_i^* >0$ for all $i=1,2,\dotsc, n$, then $m(\lambda^*)$ lies on the intersection of the boundaries of the original Mahalanobis balls.
\end{lemma}

Thus, at radius $\sqrt{\alpha}$, the intersection of the original ellipsoidal balls is non-empty. 
Therefore, by minimality, the smallest radius, $r^*$, yielding a non-empty intersection must satisfy $r^*\leqslant \sqrt{\alpha}$. 
We now show that $r^* = \sqrt{\alpha} = \sqrt{h(\lambda^*)}$.

\begin{theorem}\label{thm:mir_radius}
	Let $\lambda^*$ be a minimizer of $F$ over $\Delta^n$, and set
	$\alpha = \max_{1\leqslant i\leqslant n} \norm{\tilde c_i -m(\lambda^*)}_{A_i}^2$.
	The smallest radius that yields a non-empty intersection is
	\[
    r^* = \sqrt{\alpha}=\sqrt{h(\lambda^*)} =\sqrt{\sum_{i=1}^n \lambda_i^* \tilde c_i^\top A_i \tilde c_i -b(\lambda^*)^\top A(\lambda^*)^{-1} b(\lambda^*)}\, ,
    \]
	hence $g(\lambda^*; r^*)=0$ and so $E_{\lambda^*, r^*} =\{m(\lambda^*)\}$.
\end{theorem}

\begin{proof}
    We set $g(\lambda^*;r)=0$ and solve for $r$. This is equivalent to $\gamma(\lambda^*; r) = b(\lambda^*)^\top A(\lambda^*)^{-1} b(\lambda^*)$, or, using the expression for $\gamma(\lambda^*;r)$ above and $\sum_{i=1}^n \lambda_i^*=1$,
    \[
    r^2=\sum_{i=1}^n \lambda_i^* \tilde c_i^\top A_i \tilde c_i -b(\lambda^*)^\top A(\lambda^*)^{-1} b(\lambda^*) = h(\lambda^*),
    \]
    so $r=\sqrt{h(\lambda^*)}$. We claim this is the optimal radius $r^*$. Indeed, if $q<r$ then $g(\lambda^*;q)=h(\lambda^*)-q^2>0$, which by~\eqref{eq_cert_in_iff} implies $\displaystyle\bigcap_{i=1}^n E_q(\tilde c_i, A_i)=\emptyset$. Therefore no radius $q<r$ yields a non-empty intersection, and $r=r^*$.
	
	Now we must show that $\sqrt{h(\lambda^*)}=\sqrt{\alpha}$. By \cref{negative_obj_function_is_convex_combination_of_mahadistances}, $h(\lambda^*)=\sum_{i=1}^n \lambda_i^* \norm{\tilde c_i - m(\lambda^*)}_{A_i}^2$; by \cref{kkt_centroid} we have $\norm{\tilde c_i - m(\lambda^*)}_{A_i}^2=\alpha$ whenever $\lambda_i^*>0$, while the terms with $\lambda_i^*=0$ vanish, so $h(\lambda^*)=\alpha\sum_{i=1}^n\lambda_i^*=\alpha$ since $\lambda^*\in\Delta^n$.
\end{proof}

We summarize the results and all the notation in the following theorem.

\begin{theorem}\label{thm:mir_certificate}
	Consider $n\geqslant 1$ Mahalanobis balls at radius $r\geqslant 0$ centered at $\tilde{c}_1,\dotsc, \tilde{c}_n\in \R^d$ with orientation matrices $A_1,\dotsc, A_n \in S_{++}^d$. Let $\lambda^*\in \operatorname{argmin}_{\lambda\in\Delta^n} F(\lambda)$.
    Then we have a certificate of intersection via the criterion:
$\displaystyle\bigcap_{i=1}^n E_r(\tilde c_i, A_i) \neq \emptyset \iff g(\lambda^*; r) \leqslant 0$.
	The smallest radius that yields a non-empty intersection is
	\[
    r^* = \sqrt{\max_{1\leqslant i\leqslant n} \norm{\tilde c_i - m(\lambda^*)}_{A_i}^2}=\sqrt{-F(\lambda^*)}=\sqrt{h(\lambda^*)}\, ,
    \]
	and at this radius we have $g(\lambda^*; r^*)=0$ and the covering ellipsoid is a singleton: $E_{\lambda^*, r^*} = \{m(\lambda^*)\}$.
	The formulas remain valid in the degenerate regime $r^*=0$, which occurs precisely when all centers coincide: $E_0(\tilde c_i, A_i)=\{\tilde c_i\}$, so $\cap_i E_0(\tilde c_i, A_i)=\cap_i\{\tilde c_i\}$ is non-empty if and only if $\tilde c_1=\dotsb=\tilde c_n$. The $n=1$ case is the simplest instance ($m(\lambda^*)=\tilde c_1$ and $r^*=0$).
\end{theorem}

Thus, we can detect if the ellipsoidal balls intersect at any radius by minimizing $F(\lambda)$ over the probability simplex. 
Once that is done, we obtain the smallest radius $r^*$ by simply plugging $\lambda^*$ into one of those expressions above. 
The complexity of such a method is the sum of the complexity of the solver used to compute $\min_{\lambda \in \Delta^n} F(\lambda)$ and the complexity of the matrix arithmetic required to evaluate one of the formulas above.
For a visualization of the method, showing the behavior of the hypersurface $F(\lambda)=-h(\lambda)$, whether above or below the plane $z=0$, in relation to the existence of a common intersection, see \cref{fig:triple_none,fig:double}. 

\begin{figure}[htbp]
    \centering
    \includegraphics[scale=.25,valign=c]{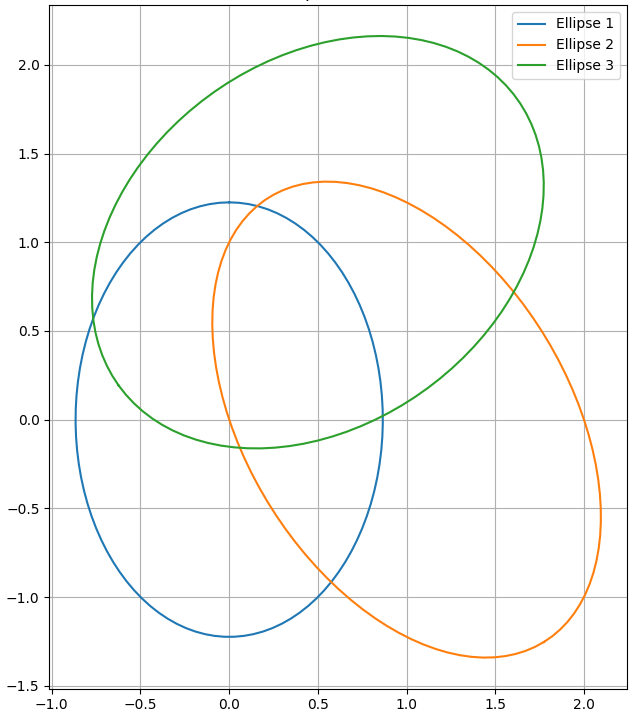}
    \includegraphics[scale=.24,valign=c]{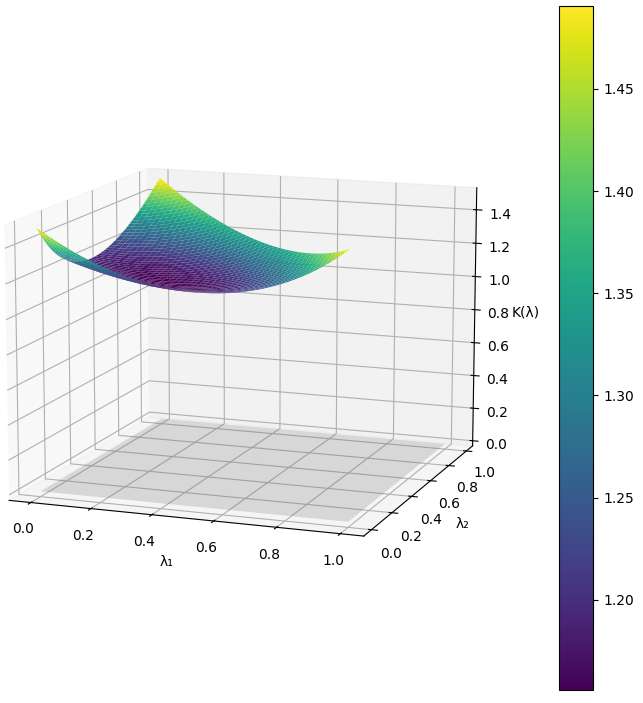}
    \includegraphics[scale=.25,valign=c]{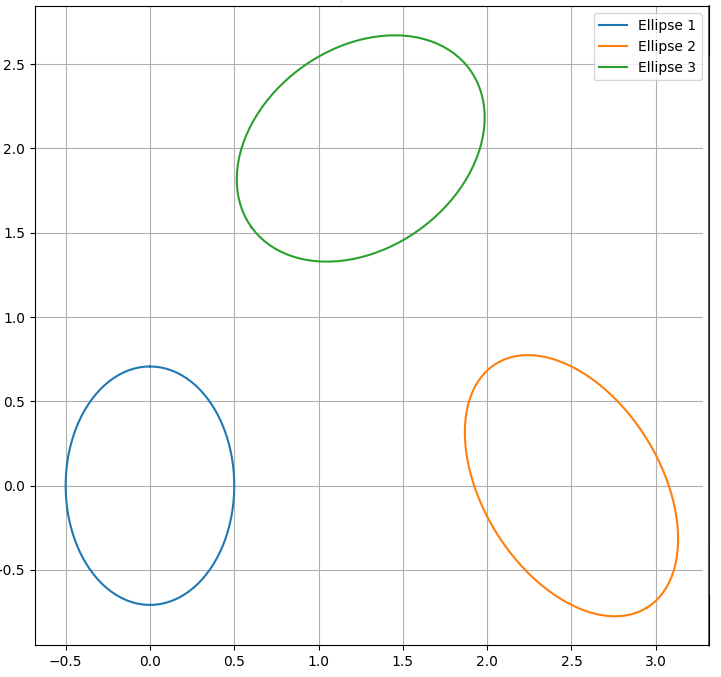}
    \includegraphics[scale=.24,valign=c]{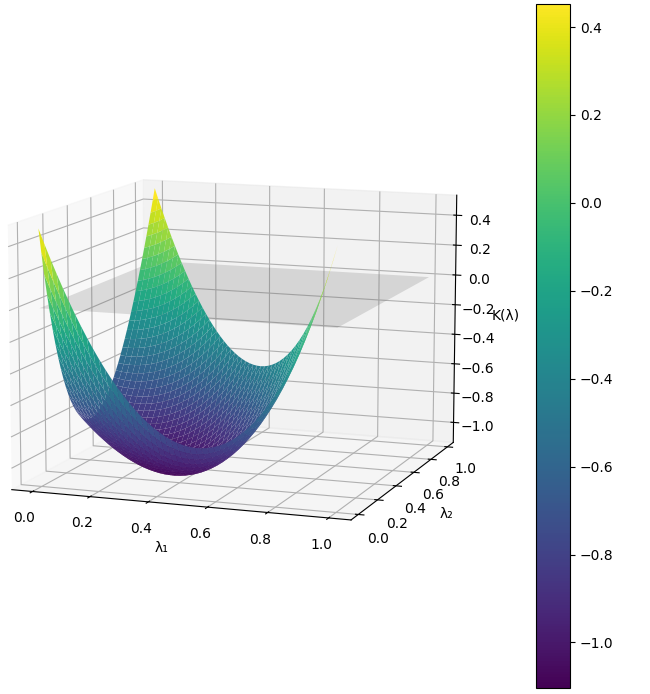}
    \caption{The hypersurface $F(\lambda)=-h(\lambda)$ for three ellipses in $\R^2$, with triple intersection (left) and no intersection (right).}
    \label{fig:triple_none}
\end{figure}

\begin{figure}[h!]
    \centering
    \includegraphics[scale=.3]{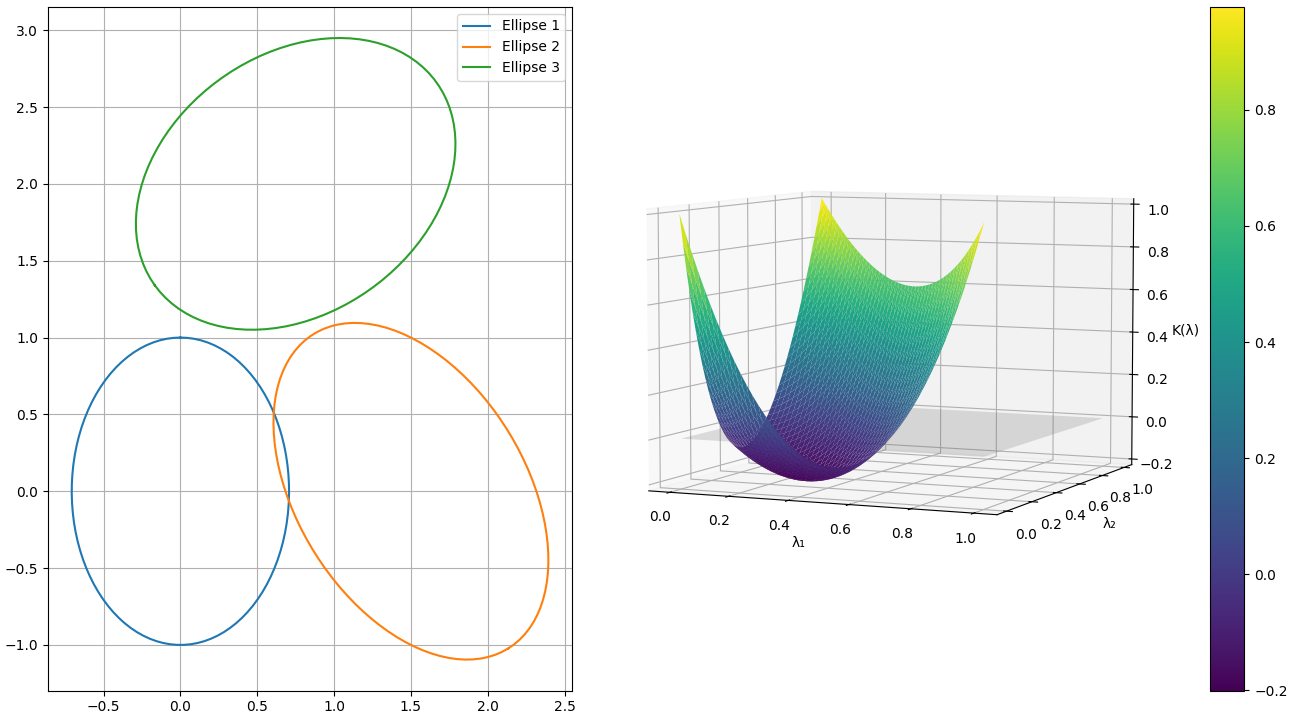}
    \caption{The hypersurface $F(\lambda)=-h(\lambda)$ for three ellipses in $\R^2$, with one pairwise intersection.}
    \label{fig:double}
\end{figure}

Note in the case of $n=2$ Mahalanobis balls, the convex optimization over the simplex can be avoided and replaced with simply counting the roots of the rational function $F(\lambda)$ on $\Delta^1 = (0,1)$ which is equivalent to counting the roots of the polynomial function $\det [A(\lambda)] F(\lambda)$ and the Leverrier algorithm can be applied to do so, as detailed in \cite{robust2ellipsoids}. 
It is not obvious if a similar workaround, avoiding the convex optimization $\min_{\lambda\in \Delta^n} F(\lambda)$, is possible in the $n>2$ case. 
The obstruction is that the $n=2$ shortcut relies on $\Delta^1=(0,1)$ being one-dimensional, so that feasibility reduces to counting the roots of the single univariate polynomial $\det[A(\lambda)]\,F(\lambda)$. For $n>2$ the optimization ranges over the higher-dimensional simplex $\Delta^{n-1}$, where no such univariate reduction is available: eliminating variables leads to a multivariate polynomial system whose resultant degrees grow rapidly with $n$ and $d$, so the Leverrier root-count no longer applies directly. We leave a full treatment of this question for future work.

We can nevertheless apply the standard pruning method: if we have $n$ Mahalanobis balls, we can run the Leverrier shortcut for each of the $\binom{n}{2}$ pairs. 
If any one of these pairs has an empty intersection, then we know the global intersection is empty. 
However, it is possible to have pair-wise intersection for all $\binom{n}{2}$ pairs but still have an empty total intersection, as described in the introduction. 

\subsection{MIR as a second-order cone program}\label{ssec_socp}

In the above section, we provided in full detail the explicit formula to obtain the MIR, together with its geometric interpretation. 
In this section, we present an alternative formulation of the MIR computation using the optimization problem for the radius as a SOCP (Second-Order Cone Program), see \cite[p. 156]{Boyd_Vandenberghe_2004}. 
While this formulation is considerably shorter, it only gives us a minimization problem that should then be solved with a solver, but no explicit geometric interpretation, which is important for our setting. 

\begin{lemma}
	Let $C = \{(t,x)\in [0,\infty)\times \R^d\colon  x\in E_t(c, A)\}$.
	Then $C$ is a convex set.
\end{lemma}

\begin{proof}
    The set $C=\{(t,x)\colon \norm{x-c}_A\leqslant t\}$ is the epigraph of the map $x\mapsto\norm{x-c}_A$, which is convex since $x\mapsto\sqrt{x^\top A x}$ is a norm (as $A$ is positive definite) precomposed with the affine map $x\mapsto x-c$; the epigraph of a convex function is convex.
\end{proof}

Since the intersection of convex sets is convex, we obtain
\begin{corollary}
	If $C_i = \{(t,x)\in [0,\infty)\times \R^d\colon  x\in E_t(c_i, A_i)\}$,
	then $C \coloneqq \{(t,x)\in [0,\infty)\times \R^d\colon x\in \cap_{i=1}^n E_t(c_i, A_i)\}=\displaystyle\bigcap_{i=1}^n C_i$
	is convex.
\end{corollary}

Thus, finding the MIR $r^*$ can be formulated as a second-order-cone-program: $r^*$ solves
\begin{align*}
	\min_{x\in\R^d,\, r\geqslant 0} \{ r \ \colon \ \norm{H_i^\top (x-\tilde c_i)}_2\leqslant r,\quad i=1,2,\dotsc,n \}
\end{align*}
where $H_i\in\R^{d\times d}$ is any matrix factor satisfying $H_iH_i^\top = A_i$ (for example, the Cholesky factor or the symmetric positive-definite square root of $A_i$). The non-negativity constraint $r\geqslant 0$ is implied by the norm bounds but is stated explicitly to match the SOCP standard form. Indeed, $\norm{H_i^\top(x-\tilde c_i)}_2^2=(x-\tilde c_i)^\top H_i H_i^\top (x-\tilde c_i)=(x-\tilde c_i)^\top A_i(x-\tilde c_i)=\norm{x-\tilde c_i}_{A_i}^2$ (independent of the choice of factor $H_i$), so the constraint $\norm{H_i^\top(x-\tilde c_i)}_2\leqslant r$ is exactly $x\in E_r(\tilde c_i, A_i)$, and the optimal value of the program is therefore the MIR $r^*$ of \hyperlink{prob}{\textbf{Q}}.

\subsection{MIR as an LP-Type problem}\label{ssec:lp} 

We now show how the problem for $r^*$ can be formulated as an LP-type problem. 
LP-type problems are well-studied~\cite{sharir_welzl, matousek_sharir_welzl, gaertner_subexp}, and there are at least three well-known algorithms for solving them, which we discuss below after specifying our specific problem.
The computation we have performed so far will be used as a certificate inside the LP-type formulation. 
Indeed, the LP setup allows us to reduce runtime by looking for the MIR only for a subset of ellipsoids, but we still need to compute the MIR of this subset.
\medskip

Let $S$ be a finite set. 
An LP-type problem involves functions $\ell\colon 2^S \to T$ where $T$ is a totally-ordered set satisfying
\begin{enumerate}
	\item Monotonicity: if $A\subset B \subset S$ then $\ell(A)\leqslant \ell(B)\leqslant \ell(S)$.
	\item Locality: Let $A\subset B \subset S$ and $i\in S$. If $\ell(A)=\ell(B)=\ell(A\cup \{i\})$ then $\ell(B) = \ell(B\cup \{i\})$. 
\end{enumerate}
Any function that satisfies these properties we call a \define{LP-type function}.

A \define{basis} of an LP-type problem is a set $B\subseteq S$ with the property that if $A\subsetneq B$ then $\ell(A)<\ell(B)$, and the \define{(combinatorial) dimension} of an LP-type problem is defined to be the maximum cardinality of a basis. 
It is assumed that an optimization algorithm may evaluate the function $\ell$ only on sets that are themselves bases or that are formed by adding a single element to a basis. 
Alternatively, the algorithm may be restricted to two primitive operations: a violation test that determines, for a basis $B$ and an element $i$, whether $\ell(B) = \ell(B \cup \{i\})$, and a basis computation that (with the same inputs) finds a basis of $B \cup \{i\}$.
The task for the algorithm to perform is to evaluate $\ell(S)$ by only using these restricted evaluations or primitives. 
In our case, $\ell(S)$ represents the optimal value of an optimization problem involving $|S|$-many constraints. 
\medskip

Let $I_n = \{1,2,\dotsc, n\}$ and define, for any non-empty $S\subset I_n$,
\begin{equation}\label{def_lp_function}
\ell(S) = \inf \{r\geqslant 0 \colon \displaystyle\bigcap_{i\in S} E_r(\tilde c_i, A_i)\neq \emptyset\}\, .
\end{equation}
For the empty set, we use the standard LP-type convention $\ell(\varnothing)=-\infty$.
Thus, $\ell(S)$ is the first radius at which the ellipsoidal balls at $c_i$ intersect for the subset of indices $i\in S$. 
We set $\ell(S)=\infty$ whenever the intersection is empty for all $r\geqslant0$, but by \cref{optimal_radius_always_finite}, this never happens.  

Since we are only concerned with the intersection of finitely many ellipsoidal balls indexed by $S\subset I_n$, it is intuitively obvious that for $r>0$ sufficiently large, the intersection will always be non-empty and hence the optimal radius $r^* = \ell(S) <\infty$.

\begin{lemma}\label{optimal_radius_always_finite}
	For any $S\subset I_n$, we have $\ell(S)<\infty$.
\end{lemma}

\begin{proof}
    If $S=\varnothing$, then $\ell(\varnothing)=-\infty<\infty$ by convention, so assume $S$ is non-empty and fix any $k\in S$.
    Setting $r=\max_{i\in S}\norm{\tilde c_k-\tilde c_i}_{A_i}$, we have $\norm{\tilde c_k-\tilde c_i}_{A_i}\leqslant r$ for every $i\in S$, so $\tilde c_k\in\bigcap_{i\in S}E_r(\tilde c_i,A_i)$.
    Hence the intersection is non-empty at radius $r$, and $\ell(S)\leqslant r<\infty$.
\end{proof}

\begin{theorem}\label{thm:lp_type}
	The function $\ell$ defined in \eqref{def_lp_function} is an LP-type function.
\end{theorem}

\begin{proof}
	By \eqref{def_lp_function} and the empty-set convention following it, $\ell$ is in the correct form, as it goes $2^{I_n}\to \R\cup\{-\infty\}$ and $\R\cup\{-\infty\}$ is a totally-ordered set.
	We first prove monotonicity. Assume $A\subset B\subset I_n$. If $A=\varnothing$, then $\ell(A)=-\infty\leqslant \ell(B)$, so assume $A$ is non-empty. Whenever $\displaystyle\bigcap_{i\in B} E_r(\tilde c_i, A_i)\neq \emptyset$, the intersection over the smaller index set satisfies $\displaystyle\bigcap_{i\in A} E_r(\tilde c_i, A_i)\supseteq\bigcap_{i\in B} E_r(\tilde c_i, A_i)\neq\emptyset$. Hence every $r$ that is feasible for $B$ is feasible for $A$, so the infimum defining $\ell(A)$ is taken over a superset of that defining $\ell(B)$, giving $\ell(A)\leqslant \ell(B)$.
	
	Next, we show locality.
    It suffices to show that, under the relevant assumptions, $\ell(B)\geqslant \ell(B\cup \{i\})$ since $\ell(B)\leqslant \ell(B\cup \{i\})$ holds by monotonicity. 
    Suppose $\ell(A)=\ell(B)=\ell(A\cup \{i\})$, and write this common value as $r$.
    Since $\ell(\varnothing)=-\infty$ and $\ell(A\cup\{i\})\geqslant 0$, these equalities imply $A\neq\varnothing$.

    We claim that $\cap_{j\in A}E_r(\tilde c_j,A_j)$ is a singleton.
    Indeed, \cref{thm:mir_radius} applies verbatim to this non-empty subfamily after relabeling its indices. Let $\lambda_A^*$ be a minimizer over the simplex for that subfamily, and let $m_A(\lambda_A^*)$ denote the corresponding center. Since $r=\ell(A)$ is the optimal radius for this subfamily, \cref{thm:mir_radius} gives
    $E_{\lambda_A^*,r}=\{m_A(\lambda_A^*)\}$
    and the intersection $\cap_{j\in A}E_r(\tilde c_j,A_j)$ is non-empty. By the first containment in \cref{covering_set_is_ellipse}, this intersection is contained in $E_{\lambda_A^*,r}$, so it must equal this singleton. Thus, for
    $x_A\coloneqq m_A(\lambda_A^*)$,
    we have $\bigcap_{j\in A}E_r(\tilde c_j,A_j)=\{x_A\}$, a point independent of the choice of minimizer $\lambda_A^*$.
    Since $A\subset B$ and $\ell(B)=r$, applying \cref{thm:mir_radius} to the subfamily indexed by $B$ shows that $\cap_{j\in B}E_r(\tilde c_j,A_j)$ is non-empty at radius $r$. This intersection is contained in $\cap_{j\in A}E_r(\tilde c_j,A_j)$; hence it contains $x_A$.
    Similarly, since $\ell(A\cup\{i\})=r$, applying \cref{thm:mir_radius} to the subfamily indexed by $A\cup\{i\}$ shows that $\cap_{j\in A\cup\{i\}}E_r(\tilde c_j,A_j)$ is non-empty at radius $r$. This intersection is contained in $\cap_{j\in A}E_r(\tilde c_j,A_j)$; hence $x_A\in E_r(\tilde c_i,A_i)$.
    Therefore $x_A\in \bigcap_{j\in B\cup\{i\}}E_r(\tilde c_j,A_j)$, so $\ell(B\cup\{i\})\leqslant r=\ell(B)$.
    Together with monotonicity, this gives $\ell(B\cup\{i\})=\ell(B)$.
\end{proof}

We can also bound the combinatorial dimension of this LP-type problem, which is what governs the cost of the algorithms below.

\begin{proposition}\label{combinatorial_dimension}
	The LP-type problem $\ell$ of \eqref{def_lp_function} has combinatorial dimension at most $d+1$.
\end{proposition}

\begin{proof}
	Let $B\subseteq I_n$ be a basis. If $\abs{B}\leqslant d+1$ we are done, so suppose $\abs{B}>d+1$. By \cref{thm:mir_radius}, which applies verbatim to the non-empty subfamily indexed by $B$ after relabeling, let $\lambda^*$ minimize $F$ over the probability simplex of $B$, and set $m^*=m(\lambda^*)$, $r=\ell(B)$; by \cref{kkt_centroid}, $\norm{\tilde c_i-m^*}_{A_i}=r$ for every $i$ in the support $T=\{i\in B:\lambda_i^*>0\}$. Since $m^*=-A(\lambda^*)^{-1}b(\lambda^*)$ and $b(\lambda^*)=-\sum_{i\in B}\lambda_i^* A_i\tilde c_i$,
	\[
	\sum_{i\in T}\lambda_i^*\, A_i(m^*-\tilde c_i)=0 ,
	\]
	which exhibits $0$ as a convex combination (as $\lambda^*\in\Delta^{B}$) of the vectors $w_i:=A_i(m^*-\tilde c_i)\in\R^d$, $i\in T$. By Carath\'eodory's theorem there are $T'\subseteq T$ with $\abs{T'}\leqslant d+1$ and weights $\beta_i\geqslant 0$, $\sum_{i\in T'}\beta_i=1$, such that $\sum_{i\in T'}\beta_i w_i=0$.

	We claim $\ell(T')=r$. Every $i\in T'$ satisfies $\norm{\tilde c_i-m^*}_{A_i}=r$, so $m^*\in\bigcap_{i\in T'}E_r(\tilde c_i,A_i)$ and thus $\ell(T')\leqslant r$. Conversely, the convex function $\phi(x)=\sum_{i\in T'}\beta_i\norm{x-\tilde c_i}_{A_i}^2$ has gradient $2\sum_{i\in T'}\beta_i A_i(x-\tilde c_i)$, which vanishes at $m^*$, so $m^*$ minimizes $\phi$ and, for every $x\in\R^d$,
	\[
	\max_{i\in T'}\norm{x-\tilde c_i}_{A_i}^2 \geqslant \phi(x) \geqslant \phi(m^*) = \sum_{i\in T'}\beta_i\norm{m^*-\tilde c_i}_{A_i}^2 = r^2 .
	\]
	Hence $\ell(T')\geqslant r$, giving $\ell(T')=r=\ell(B)$. As $\abs{T'}\leqslant d+1<\abs{B}$, the set $T'\subsetneq B$ is a proper subset with $\ell(T')=\ell(B)$, contradicting the definition of a basis. Therefore $\abs{B}\leqslant d+1$.
\end{proof}

In particular, the MIR of the whole family $I_n$ is determined by a subfamily of at most $d+1$ ellipsoids; this combinatorial-dimension bound is what makes the LP-type algorithms of \cite{clarkson,seidel,sharir_welzl,matousek_sharir_welzl,gaertner_subexp} effective on our problem.

Combining this bound with the per-subproblem cost gives an informal picture of the overall complexity. For fixed ambient dimension $d$, the LP-type algorithms evaluate $\ell(I_n)$ using an expected number of violation tests and basis computations that is linear in the number of ellipsoids $n$, the dependence on $d$ being governed by the chosen algorithm (the Matou\v sek--Sharir--Welzl analysis bounds it by a factor subexponential in the combinatorial dimension). Each such primitive solves an MIR subproblem on at most $d+1$ ellipsoids: it minimizes $F$ over a simplex of dimension at most $d$, and each inner iteration forms $A(\lambda)=\sum_i\lambda_i A_i$ (cost $O(d^2)$ per active term) and solves a $d\times d$ linear system through a Cholesky factorization (cost $O(d^3)$). Recovering $r^*$ from $\lambda^*$ via the closed-form expressions of \cref{thm:mir_certificate} again costs $O(d^3)$. Hence the method scales linearly in $n$ for fixed $d$, with the ambient dimension entering only through the dense linear algebra of each subproblem, consistent with the empirical behavior reported in \cref{sec:imp_exp}.

In \cref{ssec_convex,ssec_socp,ssec:lp}, we presented different approaches that, more or less geometrically, more or less directly, provide ways to compute the MIR of any finite family of growing, non-homogeneous ellipsoids in arbitrary ambient dimension.
Before comparing these various approaches in their practical performances (see \cref{sec:imp_exp}), we now briefly discuss why a tempting shortcut, using the minimal enclosing ellipsoid (MEE) (a technique commonly used when the ellipsoids are homogeneous), does not compute the MIR in general.

\subsection{Non-equivalence with the minimal enclosing ellipsoid}\label{ssec:diff_euclidean}

In the Euclidean case, we have that balls of radius $r>0$ with centers in the set $P$ intersect if and only if the minimal enclosing ball of $P$, MEB$(P)$, has radius at most $r$. 
When we generalize to the ellipsoidal setting, this equivalence does not hold unless all ellipsoids, including the MEE, have the same orientation matrix.
To be self-contained, we present a proof of this equivalence, even if it follows mutatis mutandis from the standard argument for the Euclidean case. 
The main result of this subsection is the proof, via counterexamples, that the equivalence fails for non-homogeneous ellipsoids.
This shows, in particular, that standard computational tools cannot in general be applied when the ellipsoids are non-homogeneous.

For a fixed orientation matrix $A$, the \define{homogeneous MEE} of a finite set $P$ is the smallest-radius ellipsoid of the form $E_\rho(y,A)$ that contains $P$.

\begin{lemma}\label{intersection mveb equivalence}
A set of $n$ homogeneous ellipsoids $E_r(x_i,A)$ of radius $r>0$ with centers in $P$ intersects if and only if the homogeneous MEE$(P)$ has radius at most $r$.
\end{lemma}

\begin{proof}
    Since the ellipsoids share the orientation $A$, the Mahalanobis norm $\norm{\cdot}_A$ is symmetric, so $\norm{y-x_i}_A\leqslant r$ for all $i$ is equivalent to $\norm{x_i-y}_A\leqslant r$ for all $i$. Hence $y\in\bigcap_{i=1}^n E_r(x_i,A)$ if and only if $E_r(y,A)$ contains every $x_i$, i.e., if and only if some radius-$r$ ellipsoid of orientation $A$ contains $P$; equivalently, the homogeneous MEE$(P)$ has radius at most $r$.
\end{proof}

Morally, the above proof works because, by homogeneity, we can transform all ellipsoids into balls. 
If at least two orientation matrices are different, there generally exists no single transformation mapping all ellipsoids into balls, and the equivalence need not hold.

\begin{proposition}\label{MEB_MIR_equivalence_fail}
For non-homogeneous ellipsoids, the equivalence of \cref{intersection mveb equivalence} can fail.
\end{proposition}

We exhibit the failure in two ways.
When $n\leq d$, the affine hull of $P$ has dimension at most $n-1<d$, so the volume infimum of any full-dimensional positive-definite enclosing ellipsoid is zero and the MEE$(P)$ is not attained in that class.
When $n\geq d+1$, finding the MEE$(P)$ can be well-posed, but it is possible to obtain an MEE of radius $r$ for a set of points $P$ whose ellipsoids intersect only at radius $r'>r$.

\begin{proof}
Let $P=\{x_1,x_2\}\in\R^2$.
Take $\ell$ as the medial axis between $x_1$ and $x_2$.
Then we can construct an infinite family of (concentric) ellipsoids passing through $x_1,x_2$, and a point $p\in\ell$ whose volumes tend to zero when $p$ tends to the mid-point between $x_1$ and $x_2$.
Hence the volume infimum among full-dimensional positive-definite ellipsoids containing $P$ is zero and no MEE$(P)$ is attained in that class.
The same flattening argument applies to any $n\leq d$, since $P$ then lies in a proper affine subspace of $\R^d$.

We now show with a counterexample that the equivalence does not hold if $n\geq d+1$. 
The construction is depicted in \cref{fig:MEB_MIR_equivalence_fail}.
Let $d=2$, and take $x_1=(0,0)$, $x_2=(1,0)$ and $x_3=(0,1)$, with orientation matrices $A_1=\begin{bmatrix} 36 & 28 \\ 28 & 36 \end{bmatrix}$, $A_2=\begin{bmatrix} 64 & 0 \\ 0 & 16 \end{bmatrix}$, and $A_3=\begin{bmatrix} 16 & 0 \\ 0 & 64 \end{bmatrix}$, respectively. 
The MIR for $P=\{(x_1,A_1),(x_2,A_2),(x_3,A_3)\}$ is $\approx 4.87$, obtained by minimizing $F(\lambda)$ over $\Delta^3$ as in \cref{thm:mir_certificate}.
However, since ellipses are convex, the MEE$(P)$ coincides with the MEE of the triangle $\operatorname{conv}\{x_1,x_2,x_3\}$, namely its Steiner circumellipse. In this example it has center $c=(\dfrac{1}{3},\dfrac{1}{3})$, orientation matrix $A = \begin{bmatrix} 3 & 3/2 \\ 3/2 & 3 \end{bmatrix}$, and radius $1$: existence and uniqueness of the minimum-area enclosing ellipse are classical (see, e.g., \cite{Behrend1938}), and for a nondegenerate triangle the unique minimizer is its Steiner circumellipse. The displayed center is the triangle centroid, and direct substitution gives $(x_i-c)^\top A(x_i-c)=1$ for each $i=1,2,3$, so $E_1(c,A)$ is precisely the MEE of the three vertices.
\end{proof}

\cref{fig:MEB_MIR_equivalence_fail} gives a visual interpretation of the lack of equivalence for $n$ non-homogeneous ellipsoids in $\mathbb{R}^d$ with $n\geq d+1$. 
We remark that, as discussed above, one could always rescale the orientation matrices of the ellipsoids to ensure that the intersection radius is the same as the MEE. 
However, this is not really solving the problem, as one would need first to compute the radius of the MEE, and then find the rescaling factor ensuring that the input ellipsoids are intersecting for that value, which is not efficient. 
Moreover, this fix does not help if $n\leq d$, since in this case the MEE is not attained in the class of full-dimensional positive-definite ellipsoids (the volume infimum is zero, as shown above).

\begin{figure}[h!]
    \centering
    \includegraphics[scale=.65]{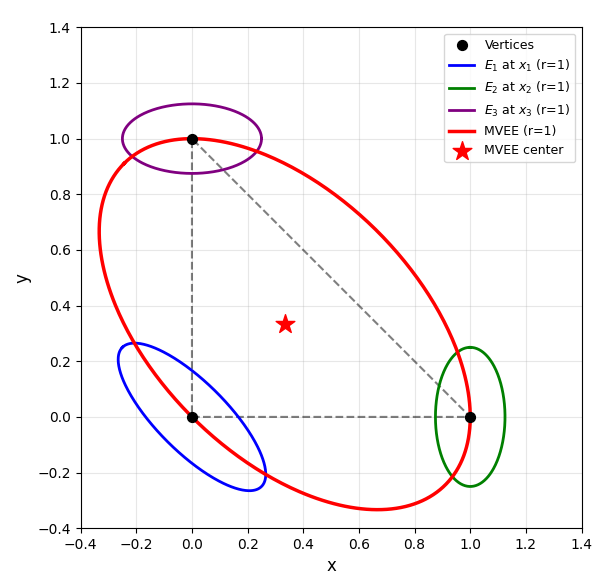}
    \includegraphics[scale=.65]{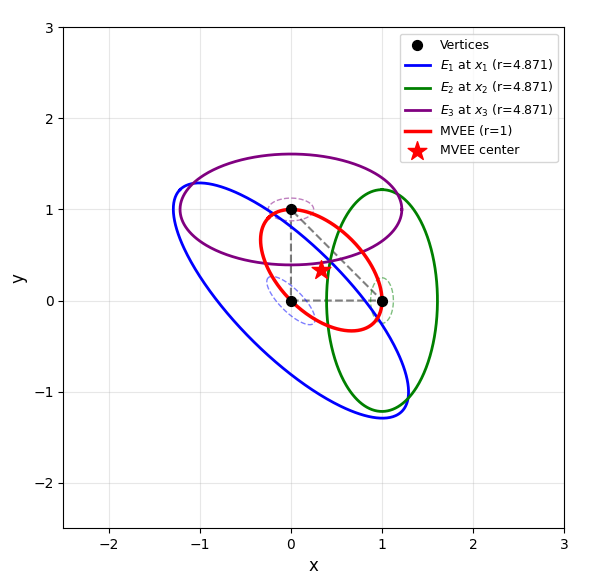}
    \caption{Graphic of the counterexample for $n=3$, $d=2$ in the proof of \cref{MEB_MIR_equivalence_fail}. The MEE (left) has radius $1$, while the three ellipsoids (at radius $1$ on the left) do not intersect until they reach radius $\approx 4.87$ (right). Note that the axes are scaled differently in the two plots.}
    \label{fig:MEB_MIR_equivalence_fail}
\end{figure}

\section{Implementation and experiments}\label{sec:imp_exp}

\subsection{Implementation details}\label{ssec:implementation}

To test the efficiency of our theory, we implemented the results of \cref{sec:mir} in \texttt{c++}. 
We provide eight implementations, divided into two categories: the \define{raw} methods, computing the MIR of $P$ directly, and the \define{LP} methods, which run the MIR computation after having found an LP-basis of $P$. 

In the former category, we used our theoretical results from \cref{ssec_convex} to implement the objective function $F(\lambda)$ and all the related matrix algebra directly. 
We then pass this to three solvers: SLSQP \cite{SLSQP}, projected gradient descent (PGD), and Cauchy \cite{chok2025convexoptimizationprobabilitysimplex}. 
SLSQP was already implemented in the \texttt{c++} library \texttt{NLopt} \cite{NLopt}, and takes as input $F(\lambda)$ with the constraints. 
On the other hand, we implemented the solvers PGD and Cauchy from scratch using the pseudocode presented in \cite{chok2025convexoptimizationprobabilitysimplex} for Cauchy, and the standard gradient descent algorithm coupled with the projection algorithm from \cite{wang2013projectionprobabilitysimplexefficient} for PGD. 
Both these solvers require only $F(\lambda)$ as an input. 
As an additional raw solver, following the results in \cref{ssec_socp}, we have used the implementation of the SOCP \cite{alglib_library} based on the Interior-Point Method \cite[p. 156]{Boyd_Vandenberghe_2004}.

In the LP category, we implemented in our setting two standard LP methods: Clarkson \cite{clarkson} and Seidel \cite{seidel}. 
These algorithms take the MIR to be the objective function $\ell(P)$ as a function of the set of indices of the points in $P$. 
As a subroutine, both methods need a raw solver to find the MIR of the subset of indices on which it is enough to compute it. 
Given its strong overall performance, particularly in the moderate-to-high ambient-dimension regime (see below), we used SLSQP for this subroutine.
We also combined both LP algorithms with the SOCP implementation. 

\subsection{Experimental results}\label{ssec:experiments}

We tested these eight implementations on randomly generated sets of ellipsoids, first fixing the ambient dimension $d$ for an increasing number $n$ of ellipsoids, then fixing $n$ for an increasing ambient dimension $d$, and finally for every pair $(n,d)$ of \# ellipsoids and ambient dimension for $n,d=2,4,8,16,32$. 
For each test, we took the average runtime of 30 randomly generated $n$ ellipsoids in $\R^d$. 
Each random instance is generated as follows.
The $n$ centers $\tilde c_i$ are drawn independently and uniformly from the cube $[-1,1]^d$.
For each ellipsoid we form a covariance $\Sigma_i = Q_i \Lambda_i Q_i^\top$, where $Q_i$ is a random orthonormal matrix obtained from the QR factorization of a matrix with i.i.d.\ standard Gaussian entries and $\Lambda_i$ is diagonal with entries drawn log-uniformly from $[\tfrac14, 4]$; the orientation matrix is then $A_i = \Sigma_i^{-1}$ and the base radius is fixed to $1$.
The same $n$ ellipsoids are passed to all eight implementations within a trial, so that runtimes and computed radii are directly comparable, and the random seed is a fixed function of $(n,d)$ and the trial index, so the benchmark is reproducible from the scripts at \cite{github_repo}.

Each timing table was produced from an independent benchmark invocation, so overlapping $(n,d)$ cells across different tables reflect separate timing measurements on bit-identical instances; the leading significant figure should be read as method-discriminating, but sub-$0.05$~ms cells can vary by tens of percent and even $\sim\!1$~ms cells by $\approx\!15\%$ between tables due to timing noise. Where two methods tie at the reported precision, the bolded entry is the one that came in strictly lower in that particular run.

We report all timing results to three significant figures.
While we did not tabulate these radii, we also verified that the computed values agreed across all eight implementations to about three significant figures, consistent with the solver tolerances. This residual spread is expected rather than a sign of error: each method reports the smallest radius at which its computed center lies in all $n$ ellipsoids, which is an upper bound on the true minimal intersection radius $r^*$ whose tightness depends on how accurately the underlying solver locates the optimal center.

All tests were executed on a 2022 MacBook Pro, Apple M2 Chip (CPU and GPU), 8 CPU cores (4 performance, 4 efficiency), 10 GPU cores, 16 GB of RAM, macOS Sequoia 15.6.1, compiled with clang++ 17.0.0, optimization flag -O3 -DNDEBUG, cmake version 4.1.2. 
Both the implementation and the scripts to generate the experiments are available at \cite{github_repo}. 

The results for the mean runtime are displayed in \cref{tab:runtime_mean_d2d3}, \cref{tab:runtime_mean_n2n3}, and \cref{tab:runtime_mean_d2d4d8d16} and \cref{tab:runtime_mean_d32_st32} (left), respectively, while the corresponding standard deviations are presented in \cref{tab:runtime_sd_d2d3}, \cref{tab:runtime_sd_n2n3}, and \cref{tab:runtime_st_d2d4d8d16} and \cref{tab:runtime_mean_d32_st32} (right). 
To better visualize the different behaviors, we plotted them in \cref{fig_fixed_d_2_3} for \cref{tab:runtime_mean_d2d3}, in \cref{fig_fixed_n_2_3} for \cref{tab:runtime_mean_n2n3}, and in \cref{fig:fixed_n24816_vs_d,fig:fixed_d24816_vs_n,fig:d32_n32} for \cref{tab:runtime_mean_d2d4d8d16} and \cref{tab:runtime_mean_d32_st32} (left).

The raw SLSQP method scales well in ambient dimension for any fixed $n$ and is the most consistent winner from $d=8$ upward across the tested values of $n$; it is also the bolded best at several low-$d$, moderate-$n$ cells in \cref{tab:runtime_mean_d2d4d8d16}. It is outperformed at $n=2$ across all tested $d$ by both LP-Clarkson variants (LP-Clarkson-SOCP is the bolded winner, with LP-Clarkson-SLSQP tying to the printed precision at $d=4,8,16$), and at low ambient dimension with large $n$ by Raw-PGD (e.g., $n=16,32$ at $d=2$ in \cref{tab:runtime_mean_d2d4d8d16}, and the $n=64,128,256$ columns at $d=2$ in \cref{tab:runtime_mean_d2d3} of \cref{app:runtime}).
The raw PGD and Cauchy methods perform similarly in moderate $(n,d)$, typically staying in the middle of the pack; at fixed $n=2,3$ they degrade sharply with $d$ and are in fact the worst-performing methods in some high-$d$ cells (e.g., Raw-Cauchy at $(n,d)=(2,128)$ and Raw-PGD at $(n,d)=(2,256)$ in \cref{tab:runtime_mean_n2n3}).
The raw SOCP implementation is never the most efficient and is often among the slower methods.

The behavior of the two LP-type algorithms is quite different.
LP-Seidel paired with the SOCP solver is the worst-performing combination in the moderate $(n,d)$ regime; at low ambient dimension with very large $n$, however, Raw-SLSQP's superlinear scaling overtakes it (e.g., Raw-SLSQP $869$~ms vs LP-Seidel-SOCP $14.3$~ms at $(n,d)=(256,2)$ in \cref{tab:runtime_mean_d2d3}).
When paired with SLSQP, LP-Seidel is more competitive and is the bolded winner for $n\geqslant 64$ at $d=2$ in \cref{tab:runtime_mean_d2d3}.

LP-Clarkson with the SOCP solver is generally slower than the SLSQP variant, but is the bolded winner at $n=2$ across the tested ambient dimensions in \cref{tab:runtime_mean_d2d4d8d16}; for $n=3$, however, the SOCP variant is competitive only at large $d$ and is several times slower than the best method at small $d$ (e.g., $0.094$~ms versus Raw-SLSQP's $0.012$~ms at $(n,d)=(3,2)$ in \cref{tab:runtime_mean_n2n3}).
Paired with SLSQP, LP-Clarkson is very efficient: it is bolded at $n=2,4$ from $d=4$ onward in \cref{tab:runtime_mean_d2d4d8d16}, often within a small factor of Raw-SLSQP and at times better.

\begin{table}[H]
\caption{Mean runtime (ms) for fixed $d=2,4,8,16$. Best performance for each $n$ is in bold. ``K'' stands for thousand.}
\label{tab:runtime_mean_d2d4d8d16}
\begin{adjustbox}{width=\columnwidth,center}
\begin{tabular}{lrrrrrrrrrr}
\toprule
\multirow{2}{*}{\diagbox[width=7.5em,innerleftsep=10pt, innerrightsep=5pt, height=2.9em]{\adjustbox{scale=.8}{Method}}{\adjustbox{scale=.8}{\# Ellipsoids}}}
& \multicolumn{5}{c}{$d=2$} & \multicolumn{5}{c}{$d=4$} \\
 \cmidrule(l{3pt}r{3pt}){2-6} \cmidrule(l{3pt}r{3pt}){7-11}
 & $n=2$ & $n=4$ & $n=8$ & $n=16$ & $n=32$  & $n=2$ & $n=4$ & $n=8$ & $n=16$ & $n=32$ \\
\midrule
Raw-SLSQP & 0.034 & \textbf{0.016} & \textbf{0.041} & 0.189 & 1.17 & 0.006 & \textbf{0.012} & \textbf{0.049} & \textbf{0.249} & 1.43 \\
Raw-PGD & 0.107 & 0.044 & 0.117 & \textbf{0.125} & \textbf{0.354} & 0.021 & 0.061 & 0.181 & 0.398 & \textbf{0.786} \\
Raw-Cauchy & 0.152 & 0.110 & 0.226 & 0.810 & 0.501 & 0.030 & 0.223 & 0.186 & 2.34 & 0.957 \\
Raw-SOCP & 0.192 & 0.122 & 0.135 & 0.240 & 0.644  & 0.081 & 0.132 & 0.303 & 0.851 & 3.43 \\
LP-Seidel-SLSQP & 0.032 & 0.064 & 0.107 & 0.243 & 0.472  & 0.013 & 0.072 & 0.263 & 1.08 & 3.45 \\
LP-Clarkson-SLSQP & 0.005 & 0.026 & 0.067 & 0.191 & 0.911  & 0.002 & 0.016 & 0.102 & 0.311 & 1.55 \\
LP-Seidel-SOCP & 0.448 & 0.744 & 1.17 & 2.34 & 4.43 & 0.213 & 0.873 & 3.22 & 10.50 & 29.90 \\
LP-Clarkson-SOCP & \textbf{0.004} & 0.173 & 0.276 & 0.389 & 0.912 & \textbf{0.002} & 0.134 & 0.693 & 2.59 & 4.84 \\
\midrule
& \multicolumn{5}{c}{$d=8$} & \multicolumn{5}{c}{$d=16$} \\
 \cmidrule(l{3pt}r{3pt}){2-6} \cmidrule(l{3pt}r{3pt}){7-11}
 & $n=2$ & $n=4$ & $n=8$ & $n=16$ & $n=32$  & $n=2$ & $n=4$ & $n=8$ & $n=16$ & $n=32$ \\
\midrule
Raw-SLSQP & 0.008 & 0.019 & \textbf{0.055} & \textbf{0.267} & \textbf{1.74}  & 0.017 & 0.035 & \textbf{0.093} & \textbf{0.385} & \textbf{2.20} \\
Raw-PGD & 0.076 & 0.184 & 0.556 & 1.49 & 2.22  & 0.218 & 0.557 & 2.48 & 6.22 & 10.20 \\
Raw-Cauchy & 0.079 & 0.66 & 0.676 & 0.957 & 3.85  & 0.131 & 0.59 & 2.34 & 3.91 & 9.07 \\
Raw-SOCP & 0.13 & 0.313 & 0.979 & 3.81 & 20.60 & 0.354 & 1.08 & 4.52 & 22.9 & 147.00 \\
LP-Seidel-SLSQP & 0.016 & 0.125 & 0.883 & 8.46 & 46.30  & 0.028 & 0.267 & 3.71 & 66.40 & 909.00 \\
LP-Clarkson-SLSQP & 0.002 & \textbf{0.014} & 0.115 & 0.571 & 2.64  & 0.002 & \textbf{0.015} & 0.256 & 1.16 & 3.85 \\
LP-Seidel-SOCP & 0.298 & 1.79 & 12.9 & 125.00 & 681.00  & 0.623 & 5.66 & 102.00 & 2.33K & 37.50K \\
LP-Clarkson-SOCP & \textbf{0.002} & 0.147 & 1.71 & 8.28 & 69.60  & \textbf{0.002} & 0.256 & 10.10 & 61.20 & 239.00 \\
\bottomrule
\end{tabular}
\end{adjustbox}
\end{table}

\begin{figure}[H]
\centering
\includegraphics[scale=.427]{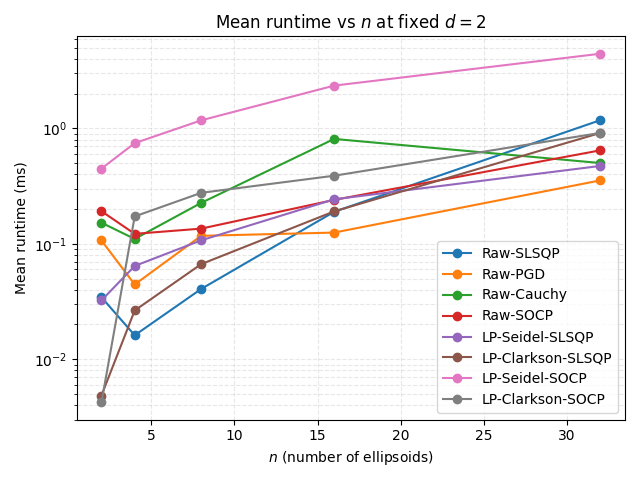}
\includegraphics[scale=.427]{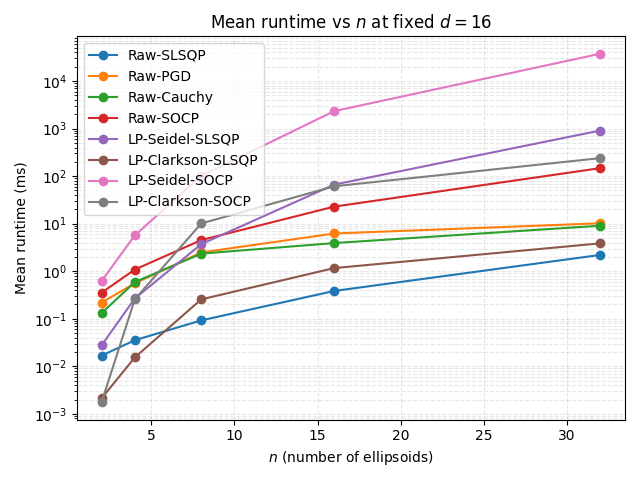}
\caption{Mean runtime for increasing number of ellipsoids $n$ at $d=2$ (left) and $d=16$ (right); the $d=4,8$ panels are in \cref{fig:fixed_d48_vs_n}.}
\label{fig:fixed_d24816_vs_n}
\end{figure}

The complete set of runtime tables and plots for every tested $(n,d)$, including the per-method heatmaps, is collected in \cref{app:runtime}; \cref{tab:runtime_mean_d2d4d8d16} and \cref{fig:fixed_d24816_vs_n} above show a representative slice.

\section{Conclusion and future directions}

In this work, we have presented three theoretical methods to compute the Minimal Intersection Radius for an arbitrarily-sized collection of growing, non-homogeneous ellipsoids in an arbitrary ambient dimension.
The first two methods are direct solvers, while the third shows that the task is an LP-type problem.
We have also implemented \cite{github_repo} our theoretical results and compared their performances on random inputs in \cref{ssec:experiments}.
We found that Raw-SLSQP and LP-Clarkson with the SLSQP certificate are the most consistently strong methods in the moderate-to-high ambient-dimension regime that motivates our intended use, so we will most likely focus on these SLSQP-based combinations in the future; LP-Clarkson-SOCP is competitive at $n=2$, and Raw-PGD is strong at low $d$ with large $n$.

These results can now be used in any procedure that requires efficient computation of the MIR of ellipsoids, for example, the construction of anisotropic \v Cech complexes, i.e., \v Cech complexes in non-Euclidean spaces, where the distances are given by ellipsoids and not balls.
For this use-case, we are especially interested in the intersection of $3$, possibly $4$ ellipsoids, since that is what the current capability of the barcode computation and the \v Cech complexes require to study not just the mutual intersections.
On the tested grid, the best methods (Raw-SLSQP and LP-Clarkson-SLSQP) stay within single-digit milliseconds (e.g., Raw-SLSQP at $(n,d)=(32,32)$ takes $3.06$~ms in \cref{tab:runtime_mean_d32_st32}), meeting the desideratum set in \cref{sec:intro}.
While we did not include the tests in this work, we also checked the runtime for the intersection of a few ellipsoids in a very high ambient dimension ($d=500$), finding that several seconds are needed.
Therefore, to use the results of this work to study manifold covering, where the ambient dimension is similarly high, we will need to further optimize our implementations, possibly drawing on additional geometric insights.

\bibliography{paper_bibliography}

\clearpage
\appendix
\section{Notation}\label{app:notation}

\cref{tab:notation} collects the main symbols used throughout \cref{sec:mir}; each is defined in detail where it first appears in the main text.

\begin{table}[htbp]
\centering
\begin{tabular}{@{}l p{0.72\columnwidth}@{}}
\toprule
Symbol & Meaning \\
\midrule
$E_r(\tilde c, A)$ & Growing ellipsoid (Mahalanobis ball) of radius $r$, center $\tilde c$, orientation matrix $A$. \\
$\tilde c_i,\ A_i$ & Center and orientation matrix of the $i$-th input ellipsoid. \\
$\norm{z}_A$ & Mahalanobis norm $\sqrt{z^\top A z}$ induced by $A\in S_{++}^d$. \\
$f_i(x;r)$ & Defining quadratic, with $E_r(\tilde c_i,A_i)=\{x: f_i(x;r)\leqslant 0\}$. \\
$\Delta^n$ & Probability simplex $\{\lambda\geqslant 0:\ \textstyle\sum_i\lambda_i=1\}$. \\
$A(\lambda),\ b(\lambda)$ & $\sum_i\lambda_i A_i$ and $\sum_i\lambda_i b_i$, where $b_i=-A_i\tilde c_i$. \\
$m(\lambda)$ & Center $-A(\lambda)^{-1}b(\lambda)$ of the Lagrangian covering ellipsoid. \\
$E_{\lambda,r}$ & Lagrangian covering ellipsoid (\cref{covering_set_is_ellipse}). \\
$\gamma_i(r),\ \gamma(\lambda;r)$ & $\tilde c_i^\top A_i\tilde c_i-r^2$ and $\sum_i\lambda_i\gamma_i(r)$. \\
$g(\lambda;r)$ & Dual function $\gamma(\lambda;r)-b(\lambda)^\top A(\lambda)^{-1}b(\lambda)$. \\
$h(\lambda),\ F(\lambda)$ & $h(\lambda)=\sum_i\lambda_i\norm{\tilde c_i-m(\lambda)}_{A_i}^2$ and objective $F=-h$. \\
$\ell(S)$ & MIR of the subfamily indexed by $S$ (the LP-type set function). \\
$r^*$ & The minimal intersection radius (MIR). \\
\bottomrule
\end{tabular}
\caption{Main notation used throughout \cref{sec:mir}.}
\label{tab:notation}
\end{table}

\section{Detailed runtime results}\label{app:runtime}

For completeness, we collect here the full runtime tables and plots for all tested $(n,d)$.

\begin{table}[h!]
\caption{Mean runtime (ms) for fixed $d=2$ (top) and $d=3$ (bottom). Best performance for each $n$ is in bold.}
\label{tab:runtime_mean_d2d3}
\begin{adjustbox}{max width=\columnwidth,center}
\begin{tabular}{lrrrrrrrr}
\toprule
\multirow{2}{*}{\diagbox[width=7.5em,innerleftsep=10pt, innerrightsep=5pt, height=2.9em]{\adjustbox{scale=.8}{Method}}{\adjustbox{scale=.8}{\# Ellipsoids}}}
& \multicolumn{8}{c}{$d=2$} \\
 \cmidrule(l{3pt}r{3pt}){2-9}
 & $n=2$ & $n=4$ & $n=8$ & $n=16$ & $n=32$ & $n=64$ & $n=128$ & $n=256$ \\
\midrule
Raw-SLSQP & 0.045 & \textbf{0.017} & \textbf{0.041} & 0.197 & 1.01 & 8.47 & 74.80 & 869.00 \\
Raw-PGD & 0.109 & 0.0443 & 0.125 & \textbf{0.139} & \textbf{0.380} & 1.19 & 2.58 & 3.90 \\
Raw-Cauchy & 0.175 & 0.118 & 0.255 & 0.987 & 0.60 & 1.92 & 6.91 & 17.40 \\
Raw-SOCP & 0.201 & 0.125 & 0.138 & 0.262 & 0.652 & 2.15 & 9.18 & 51.10 \\
LP-Seidel-SLSQP & 0.035 & 0.068 & 0.122 & 0.291 & 0.559 & \textbf{1.07} & \textbf{2.31} & \textbf{2.73} \\
LP-Clarkson-SLSQP & 0.005 & 0.029 & 0.071 & 0.208 & 0.849 & 4.88 & 11.70 & 24.20 \\
LP-Seidel-SOCP & 0.461 & 0.770 & 1.25 & 2.57 & 4.45 & 7.31 & 13.60 & 14.30 \\
LP-Clarkson-SOCP & \textbf{0.004} & 0.183 & 0.281 & 0.416 & 0.826 & 2.16 & 3.57 & 5.69 \\
\midrule
& \multicolumn{8}{c}{$d=3$} \\
 \cmidrule(l{3pt}r{3pt}){2-9}
Raw-SLSQP & 0.006 & \textbf{0.013} & \textbf{0.043} & 0.223 & 1.22 & 9.64 & 94.60 & 913.00 \\
Raw-PGD & 0.021 & 0.069 & 0.207 & \textbf{0.211} & \textbf{0.652} & \textbf{1.73} & \textbf{4.36} & 10.20 \\
Raw-Cauchy & 0.025 & 0.067 & 0.573 & 0.722 & 1.69 & 5.44 & 11.90 & \textbf{7.39} \\
Raw-SOCP & 0.081 & 0.112 & 0.195 & 0.468 & 1.41 & 6.47 & 36.20 & 257.00 \\
LP-Seidel-SLSQP & 0.015 & 0.053 & 0.191 & 0.524 & 1.23 & 3.05 & 6.65 & 13.80 \\
LP-Clarkson-SLSQP & 0.003 & 0.021 & 0.077 & 0.279 & 0.99 & 6.78 & 52.20 & 101.00 \\
LP-Seidel-SOCP & 0.227 & 0.649 & 1.87 & 4.59 & 9.54 & 21.00 & 41.30 & 72.40 \\
LP-Clarkson-SOCP & \textbf{0.002} & 0.144 & 0.378 & 0.894 & 1.67 & 6.41 & 58.00 & 53.80 \\
\bottomrule
\end{tabular}
\end{adjustbox}
\end{table}

\begin{figure}[h!]
\centering
\includegraphics[scale=.451]{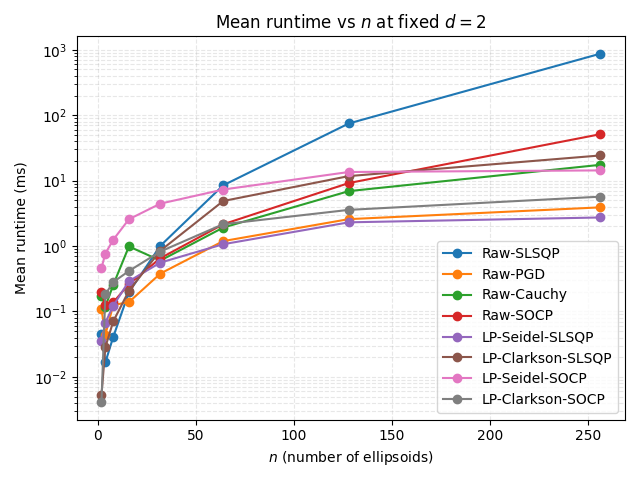}
\includegraphics[scale=.451]{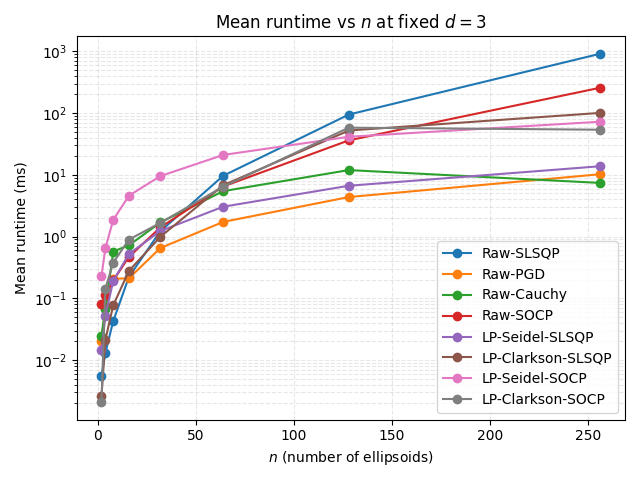}
\caption{Mean runtime for increasing $n$ for fixed $d=2,3$.}
\label{fig_fixed_d_2_3}
\end{figure}

\begin{table}
\caption{Mean runtime (ms) for fixed $n=2$ (top) and $n=3$ (bottom). Best performance for each $d$ is in bold.}
\label{tab:runtime_mean_n2n3}
\begin{adjustbox}{max width=\columnwidth,center}
\begin{tabular}{lrrrrrrrr}
\toprule
\multirow{2}{*}{\diagbox[width=7.5em,innerleftsep=10pt, innerrightsep=5pt, height=2.9em]{\adjustbox{scale=.8}{Method}}{\adjustbox{scale=.8}{Dimension}}}
& \multicolumn{8}{c}{$n=2$} \\
 \cmidrule(l{3pt}r{3pt}){2-9}
 & $d=2$ & $d=4$ & $d=8$ & $d=16$ & $d=32$ & $d=64$ & $d=128$ & $d=256$ \\
\midrule
Raw-SLSQP & 0.047 & 0.007 & 0.010 & 0.017 & 0.048 & 0.189 & 0.746 & 4.25 \\
Raw-PGD & 0.114 & 0.027 & 0.086 & 0.219 & 1.02 & 5.93 & 12.3 & 438.00 \\
Raw-Cauchy & 0.168 & 0.039 & 0.087 & 0.135 & 0.947 & 4.82 & 241 & 210 \\
Raw-SOCP & 0.196 & 0.105 & 0.142 & 0.352 & 1.28 & 6.14 & 34.3 & 226 \\
LP-Seidel-SLSQP & 0.036 & 0.018 & 0.019 & 0.030 & 0.074 & 0.269 & 1.01 & 5.51 \\
LP-Clarkson-SLSQP & 0.005 & 0.003 & 0.003 & \textbf{0.003} & 0.004 & 0.006 & 0.027 & 0.096 \\
LP-Seidel-SOCP & 0.468 & 0.274 & 0.326 & 0.628 & 1.89 & 8.25 & 42.40 & 265.00 \\
LP-Clarkson-SOCP & \textbf{0.004} & \textbf{0.003} & \textbf{0.002} & 0.0033 & \textbf{0.003} & \textbf{0.006} & \textbf{0.026} & \textbf{0.091} \\
\midrule
& \multicolumn{8}{c}{$n=3$} \\
 \cmidrule(l{3pt}r{3pt}){2-9}
Raw-SLSQP & \textbf{0.012} & 0.011 & 0.013 & 0.029 & 0.085 & 0.338 & 1.31 & 6.26 \\
Raw-PGD & 0.079 & 0.049 & 0.208 & 0.487 & 3.99 & 4.92 & 18.60 & 108.00 \\
Raw-Cauchy & 0.193 & 0.075 & 0.116 & 0.48 & 1.07 & 4.79 & 23.90 & 65.60 \\
Raw-SOCP & 0.114 & 0.118 & 0.212 & 0.689 & 2.59 & 13.9 & 92.8 & 689.00 \\
LP-Seidel-SLSQP & 0.038 & 0.043 & 0.052 & 0.109 & 0.297 & 1.09 & 4.26 & 21.40 \\
LP-Clarkson-SLSQP & 0.015 & \textbf{0.006} & \textbf{0.008} & \textbf{0.004} & 0.005 & 0.009 & 0.040 & 0.14 \\
LP-Seidel-SOCP & 0.478 & 0.571 & 0.752 & 2.24 & 7.38 & 35.3 & 208.00 & 1450.0 \\
LP-Clarkson-SOCP & 0.094 & 0.038 & 0.070 & 0.028 & \textbf{0.004} & \textbf{0.008} & \textbf{0.038} & \textbf{0.137} \\
\bottomrule
\end{tabular}
\end{adjustbox}
\end{table}

\begin{figure}[htbp]
\centering
\includegraphics[scale=.451]{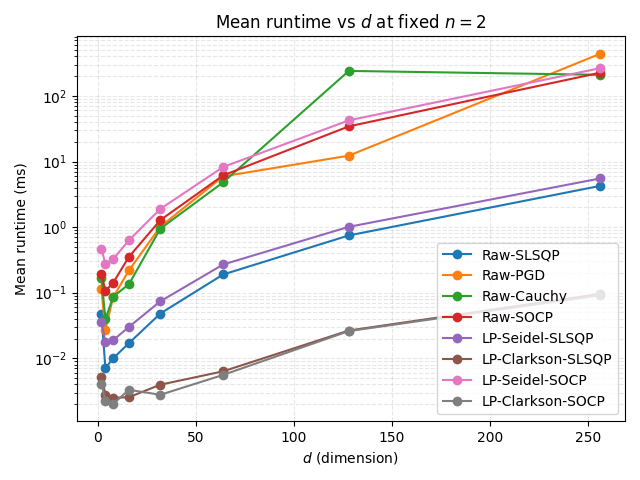}
\includegraphics[scale=.451]{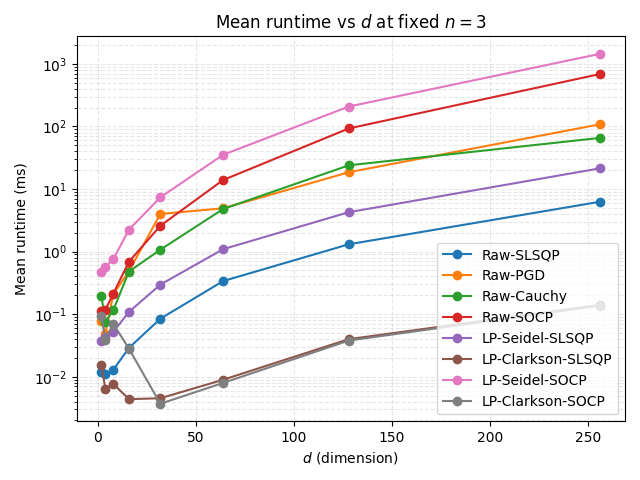}
\caption{Mean runtime for increasing $d$ for fixed $n=2,3$.}
\label{fig_fixed_n_2_3}
\end{figure}

\begin{table}
\caption{Mean runtime (ms) and standard deviation runtime for fixed $d=32$. Best mean runtime performance for each $n$ is in bold. ``K'' stands for thousand.}
\label{tab:runtime_mean_d32_st32}
\begin{adjustbox}{width=\columnwidth,center}
\begin{tabular}{lrrrrrrrrrr}
\toprule
\multirow{2}{*}{\diagbox[width=7.5em,innerleftsep=10pt, innerrightsep=5pt, height=2.9em]{\adjustbox{scale=.8}{Method}}{\adjustbox{scale=.8}{\# Ellipsoids}}}
& \multicolumn{5}{c}{Mean} & \multicolumn{5}{c}{Standard deviation} \\
 \cmidrule(l{3pt}r{3pt}){2-6} \cmidrule(l{3pt}r{3pt}){7-11}
 & $n=2$ & $n=4$ & $n=8$ & $n=16$ & $n=32$  & $n=2$ & $n=4$ & $n=8$ & $n=16$ & $n=32$ \\
\midrule
Raw-SLSQP & 0.0477 & 0.115 & \textbf{0.241} & \textbf{0.769} & \textbf{3.06}  & 0.0144 & 0.0181 & 0.0298 & 0.156 & 0.602 \\
Raw-PGD & 1.03 & 1.17 & 14.40 & 36.90 & 84.20 & 1.55 & 1.24 & 7.60 & 11.90 & 14.40 \\
Raw-Cauchy & 0.959 & 2.51 & 6.79 & 15.5 & 65.30 & 1.58 & 3.60 & 15.40 & 19.20 & 109.00 \\
Raw-SOCP & 1.27 & 4.68 & 23.50 & 146.00 & 1.09K & 0.006 & 0.635 & 2.86 & 11.70 & 100.00 \\
LP-Seidel-SLSQP & 0.072 & 0.886 & 21.70 & 507.00 & 13.1K & 0.015 & 0.056 & 6.86 & 299.00 & 10.8K \\
LP-Clarkson-SLSQP & 0.003 & 0.011 & 0.524 & 3.66 & 12.80 & 0.000 & 0.002 & 0.244 & 1.28 & 7.78 \\
LP-Seidel-SOCP & 1.91 & 24.80 & 883.00 & 30.7K & 1140K  & 0.077 & 1.40 & 310.00 & 20.5K & 1010K \\
LP-Clarkson-SOCP & \textbf{0.003} & \textbf{0.009} & 39.30 & 548.00 & 6.78K  & 0.000 & 0.001 & 20.60 & 271.00 & 11.4K \\
\bottomrule
\end{tabular}
\end{adjustbox}
\end{table}

\begin{table}
\caption{Runtime standard deviation (ms) for fixed $d=2$ (top) and $d=3$ (bottom).}
\label{tab:runtime_sd_d2d3}
\begin{adjustbox}{max width=\columnwidth,center}
\begin{tabular}{lrrrrrrrr}
\toprule
\multirow{2}{*}{\diagbox[width=7.5em,innerleftsep=10pt, innerrightsep=5pt, height=2.9em]{\adjustbox{scale=.8}{Method}}{\adjustbox{scale=.8}{\# Ellipsoids}}}
& \multicolumn{8}{c}{$d=2$} \\
 \cmidrule(l{3pt}r{3pt}){2-9} & $n=2$ & $n=4$ & $n=8$ & $n=16$ & $n=32$ & $n=64$ & $n=128$ & $n=256$ \\
\midrule
Raw-SLSQP & 0.177 & 0.005 & 0.008 & 0.052 & 0.193 & 2.11 & 20.20 & 888.00 \\
Raw-PGD & 0.252 & 0.035 & 0.245 & 0.106 & 0.203 & 1.04 & 3.93 & 2.98 \\
Raw-Cauchy & 0.483 & 0.188 & 0.492 & 3.26 & 0.75 & 4.31 & 27.2 & 53.80 \\
Raw-SOCP & 0.219 & 0.023 & 0.015 & 0.039 & 0.045 & 0.229 & 0.456 & 2.42 \\
LP-Seidel-SLSQP & 0.014 & 0.028 & 0.063 & 0.15 & 0.328 & 0.784 & 1.43 & 2.05 \\
LP-Clarkson-SLSQP & 0.001 & 0.015 & 0.036 & 0.107 & 0.428 & 2.2 & 5.32 & 9.54 \\
LP-Seidel-SOCP & 0.125 & 0.284 & 0.594 & 1.19 & 2.36 & 4.67 & 7.87 & 10.20 \\
LP-Clarkson-SOCP & 0.001 & 0.100 & 0.14 & 0.201 & 0.412 & 0.911 & 1.40 & 1.48 \\
\midrule
& \multicolumn{8}{c}{$d=3$} \\
 \cmidrule(l{3pt}r{3pt}){2-9}
Raw-SLSQP & 0.001 & 0.003 & 0.008 & 0.059 & 0.384 & 2.51 & 31.6 & 237.00 \\
Raw-PGD & 0.017 & 0.118 & 0.388 & 0.183 & 0.457 & 0.634 & 3.86 & 8.64 \\
Raw-Cauchy & 0.042 & 0.129 & 2.30 & 1.53 & 4.08 & 16.40 & 34.10 & 11.10 \\
Raw-SOCP & 0.018 & 0.013 & 0.018 & 0.027 & 0.080 & 0.385 & 3.46 & 22.10 \\
LP-Seidel-SLSQP & 0.002 & 0.025 & 0.104 & 0.328 & 0.857 & 3.28 & 4.73 & 9.83 \\
LP-Clarkson-SLSQP & 0.000 & 0.016 & 0.040 & 0.126 & 0.531 & 2.93 & 20.20 & 40.70 \\
LP-Seidel-SOCP & 0.029 & 0.291 & 0.890 & 2.64 & 6.28 & 20.8 & 27.90 & 45.50 \\
LP-Clarkson-SOCP & 0.000 & 0.129 & 0.185 & 0.37 & 0.897 & 2.48 & 147.00 & 17.40 \\
\bottomrule
\end{tabular}
\end{adjustbox}
\end{table}

\begin{table}
\caption{Runtime standard deviation (ms) for fixed $n=2$ (top) and $n=3$ (bottom).}
\label{tab:runtime_sd_n2n3}
\begin{adjustbox}{max width=\columnwidth,center}
\begin{tabular}{lrrrrrrrr}
\toprule
\multirow{2}{*}{\diagbox[width=7.5em,innerleftsep=10pt, innerrightsep=5pt, height=2.9em]{\adjustbox{scale=.8}{Method}}{\adjustbox{scale=.8}{Dimension}}}
& \multicolumn{8}{c}{$n=2$} \\
 \cmidrule(l{3pt}r{3pt}){2-9} & $d=2$ & $d=4$ & $d=8$ & $d=16$ & $d=32$ & $d=64$ & $d=128$ & $d=256$ \\
\midrule
Raw-SLSQP & 0.180 & 0.002 & 0.009 & 0.00415 & 0.014 & 0.071 & 0.242 & 1.10 \\
Raw-PGD & 0.274 & 0.020 & 0.083 & 0.421 & 1.53 & 8.48 & 10.20 & 931 \\
Raw-Cauchy & 0.478 & 0.104 & 0.234 & 0.218 & 1.56 & 6.23 & 984.00 & 330.00 \\
Raw-SOCP & 0.208 & 0.011 & 0.021 & 0.024 & 0.071 & 0.218 & 0.996 & 7.76 \\
LP-Seidel-SLSQP & 0.014 & 0.002 & 0.004 & 0.005 & 0.015 & 0.069 & 0.238 & 1.09 \\
LP-Clarkson-SLSQP & 0.001 & 0.000 & 0.000 & 0.000 & 0.004 & 0.000 & 0.001 & 0.009 \\
LP-Seidel-SOCP & 0.134 & 0.044 & 0.052 & 0.054 & 0.074 & 0.358 & 1.35 & 8.91 \\
LP-Clarkson-SOCP & 0.001 & 0.000 & 0.000 & 0.007 & 0.000 & 0.000 & 0.001 & 0.001 \\
\midrule
& \multicolumn{8}{c}{$n=3$} \\
 \cmidrule(l{3pt}r{3pt}){2-9}
 Raw-SLSQP & 0.003 & 0.003 & 0.002 & 0.006 & 0.015 & 0.066 & 0.271 & 1.06 \\
Raw-PGD & 0.284 & 0.046 & 0.635 & 1.29 & 9.48 & 5.92 & 9.11 & 51.20 \\
Raw-Cauchy & 0.40 & 0.085 & 0.196 & 1.27 & 0.941 & 6.71 & 58.70 & 62.50 \\
Raw-SOCP & 0.024 & 0.012 & 0.015 & 0.076 & 0.174 & 0.761 & 5.80 & 40.70 \\
LP-Seidel-SLSQP & 0.013 & 0.009 & 0.018 & 0.018 & 0.038 & 0.141 & 0.543 & 2.60 \\
LP-Clarkson-SLSQP & 0.007 & 0.005 & 0.008 & 0.004 & 0.000 & 0.001 & 0.001 & 0.004 \\
LP-Seidel-SOCP & 0.138 & 0.126 & 0.217 & 0.3 & 0.254 & 0.839 & 6.05 & 66.10 \\
LP-Clarkson-SOCP & 0.048 & 0.056 & 0.106 & 0.139 & 0.000 & 0.000 & 0.000 & 0.005 \\
\bottomrule
\end{tabular}
\end{adjustbox}
\end{table}

\begin{table}
\caption{Standard deviation runtime (ms) for fixed $d=2,4,8,16$. ``K'' stands for thousand.}
\label{tab:runtime_st_d2d4d8d16}
\begin{adjustbox}{width=\columnwidth,center}
\begin{tabular}{lrrrrrrrrrr}
\toprule
\multirow{2}{*}{\diagbox[width=7.5em,innerleftsep=10pt, innerrightsep=5pt, height=2.9em]{\adjustbox{scale=.8}{Method}}{\adjustbox{scale=.8}{\# Ellipsoids}}}
& \multicolumn{5}{c}{$d=2$} & \multicolumn{5}{c}{$d=4$} \\
 \cmidrule(l{3pt}r{3pt}){2-6} \cmidrule(l{3pt}r{3pt}){7-11}
 & $n=2$ & $n=4$ & $n=8$ & $n=16$ & $n=32$  & $n=2$ & $n=4$ & $n=8$ & $n=16$ & $n=32$ \\
\midrule
Raw-SLSQP & 0.116 & 0.004 & 0.009 & 0.048 & 0.859 & 0.002 & 0.002 & 0.018 & 0.050 & 0.444 \\
Raw-PGD & 0.257 & 0.036 & 0.217 & 0.095 & 0.191 & 0.016 & 0.047 & 0.181 & 0.435 & 0.385 \\
Raw-Cauchy & 0.437 & 0.172 & 0.417 & 2.69 & 0.633 & 0.083 & 0.923 & 0.217 & 6.20 & 1.31 \\
Raw-SOCP & 0.195 & 0.020 & 0.014 & 0.011 & 0.056 & 0.004 & 0.011 & 0.045 & 0.085 & 0.784 \\
LP-Seidel-SLSQP & 0.011 & 0.028 & 0.057 & 0.124 & 0.291 & 0.006 & 0.030 & 0.164 & 0.902 & 2.75 \\
LP-Clarkson-SLSQP & 0.002 & 0.014 & 0.035 & 0.098 & 0.548 & 0.000 & 0.013 & 0.046 & 0.142 & 0.712 \\
LP-Seidel-SOCP & 0.122 & 0.279 & 0.561 & 1.11 & 3.02 & 0.028 & 0.313 & 2.05 & 7.92 & 23.00 \\
LP-Clarkson-SOCP & 0.002 & 0.096 & 0.138 & 0.188 & 0.684 & 0.00 & 0.129 & 0.357 & 6.13 & 2.63 \\
\midrule
& \multicolumn{5}{c}{$d=8$} & \multicolumn{5}{c}{$d=16$} \\
 \cmidrule(l{3pt}r{3pt}){2-6} \cmidrule(l{3pt}r{3pt}){7-11}
 & $n=2$ & $n=4$ & $n=8$ & $n=16$ & $n=32$  & $n=2$ & $n=4$ & $n=8$ & $n=16$ & $n=32$ \\
\midrule
Raw-SLSQP & 0.002 & 0.005 & 0.007 & 0.044 & 0.338 & 0.004 & 0.008 & 0.011 & 0.061 & 0.512 \\
Raw-PGD & 0.077 & 0.173 & 0.576 & 0.893 & 1.21 & 0.432 & 0.627 & 1.27 & 1.60 & 2.89 \\
Raw-Cauchy & 0.214 & 2.11 & 1.16 & 1.15 & 5.89 & 0.204 & 0.78 & 5.78 & 6.80 & 15.40 \\
Raw-SOCP & 0.004 & 0.034 & 0.076 & 0.290 & 1.99 & 0.017 & 0.082 & 0.497 & 2.20 & 11.50 \\
LP-Seidel-SLSQP & 0.002 & 0.043 & 0.55 & 10.20 & 37.50 & 0.004 & 0.068 & 1.64 & 54.40 & 924.00 \\
LP-Clarkson-SLSQP & 0.000 & 0.016 & 0.065 & 0.209 & 1.19 & 0.000 & 0.025 & 0.077 & 0.538 & 1.53 \\
LP-Seidel-SOCP & 0.024 & 0.588 & 8.01 & 153 & 547 & 0.047 & 1.26 & 48.10 & 2.06K & 40.8K \\
LP-Clarkson-SOCP & 0.00 & 0.255 & 1.05 & 2.94 & 138.00 & 0.000 & 0.699 & 3.31 & 30.30 & 90.20 \\
\bottomrule
\end{tabular}
\end{adjustbox}
\end{table}

\begin{figure}[htbp]
\centering
\includegraphics[scale=.451]{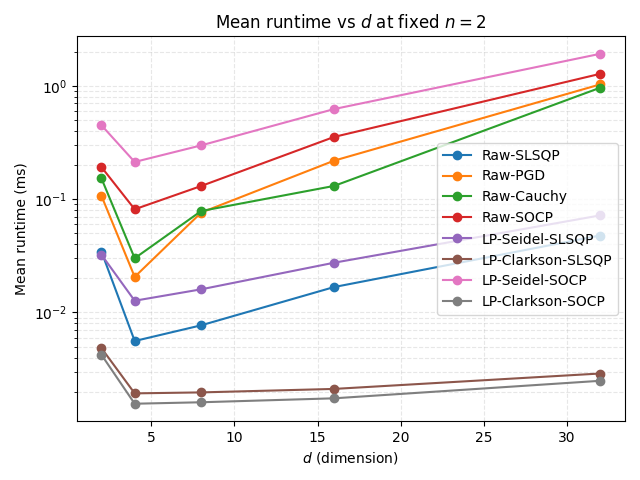}
\includegraphics[scale=.451]{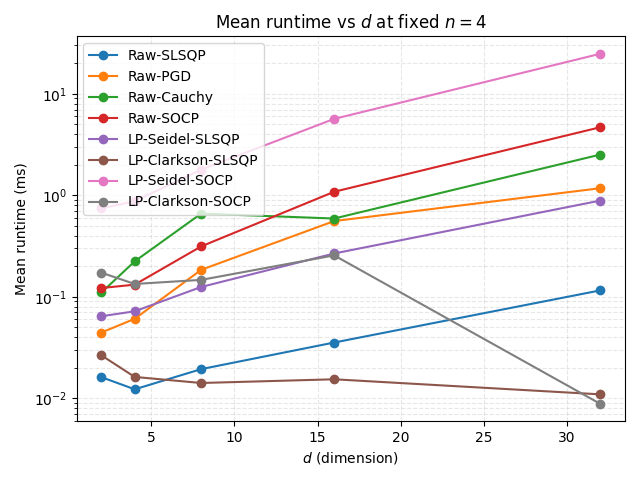}
\includegraphics[scale=.451]{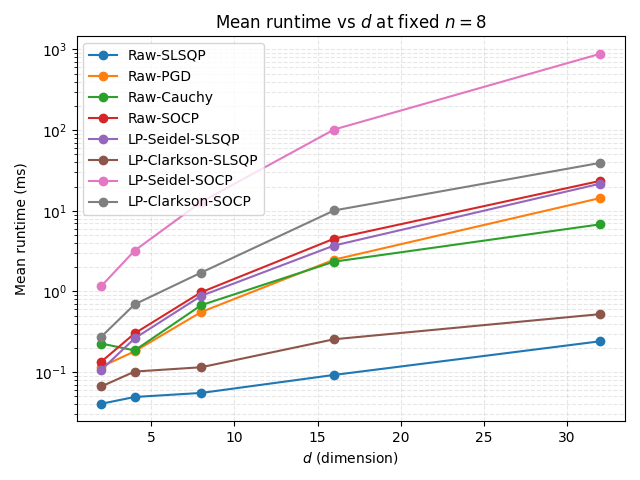}
\includegraphics[scale=.451]{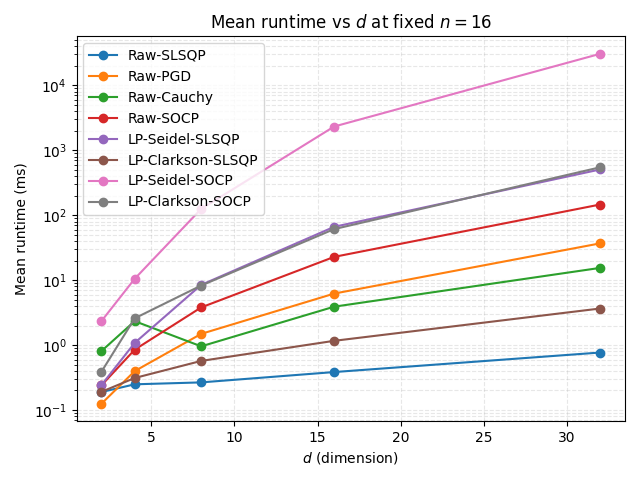}
\caption{Mean runtime for increasing ambient dimension $d$ for fixed $n=2,4$ (top) and $n=8,16$ (bottom).}
\label{fig:fixed_n24816_vs_d}
\end{figure}

\begin{figure}[htbp]
\centering
\includegraphics[scale=.427]{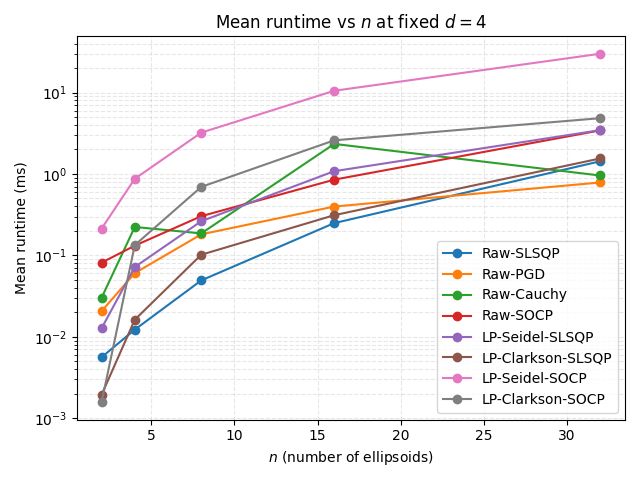}
\includegraphics[scale=.427]{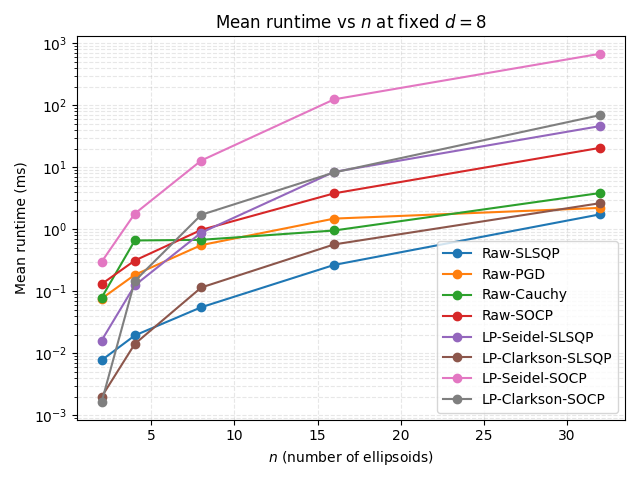}
\caption{Mean runtime for increasing number of ellipsoids $n$ at $d=4$ (left) and $d=8$ (right); see also \cref{fig:fixed_d24816_vs_n}.}
\label{fig:fixed_d48_vs_n}
\end{figure}

\begin{figure}[htbp]
\centering
\includegraphics[scale=.427]{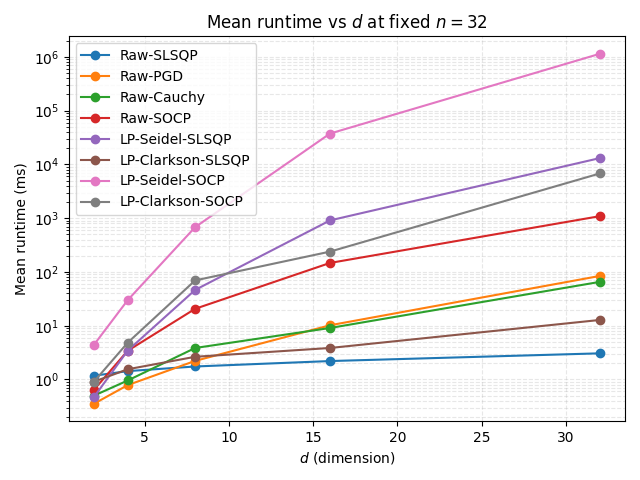}
\includegraphics[scale=.427]{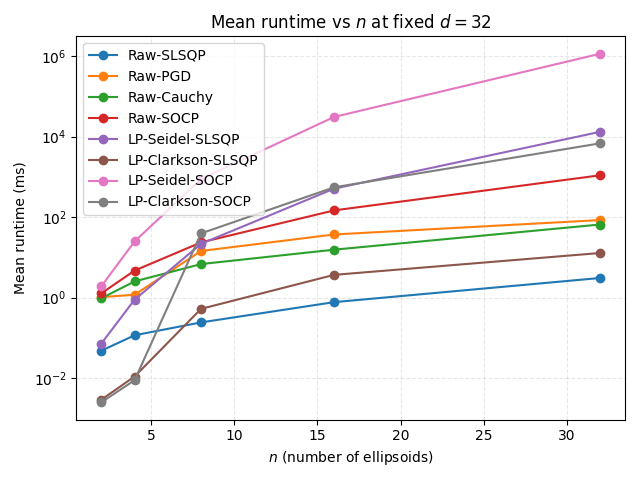}
\caption{Mean runtime for increasing ambient dimension $d$ for fixed $n=32$ (left) and for increasing number of ellipsoids $n$ for fixed $d=32$ (right).}
\label{fig:d32_n32}
\end{figure}

\begin{figure}[htbp]
\centering
\includegraphics[scale=.427]{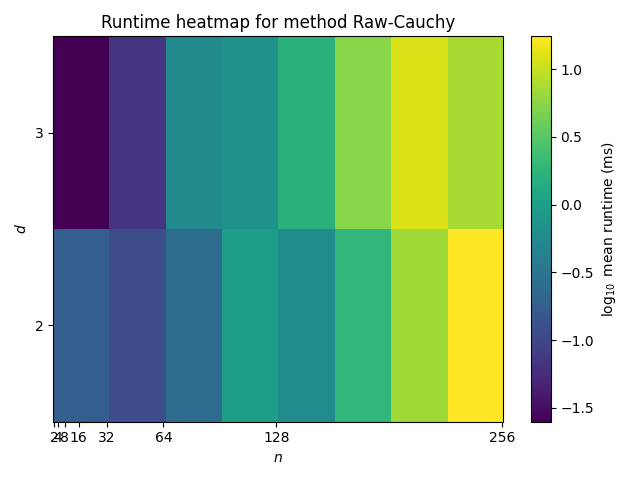}
\includegraphics[scale=.427]{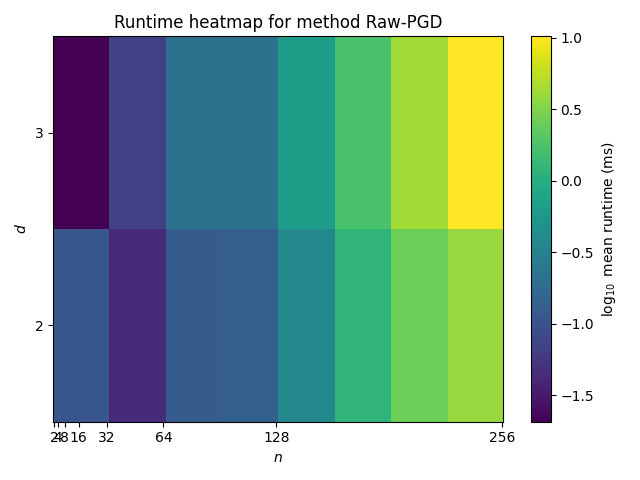}
\caption{Heatmaps of the mean runtime for the raw methods Cauchy (left) and PGD (right), for $d=2,3$ and increasing $n$. In each plot, the horizontal axis is the number of ellipsoids, and the vertical axis is the ambient dimension.}
\label{fig_heatmaps_cauchy_pgd_d}
\end{figure}

\begin{figure}[htbp]
\centering
\includegraphics[scale=.299]{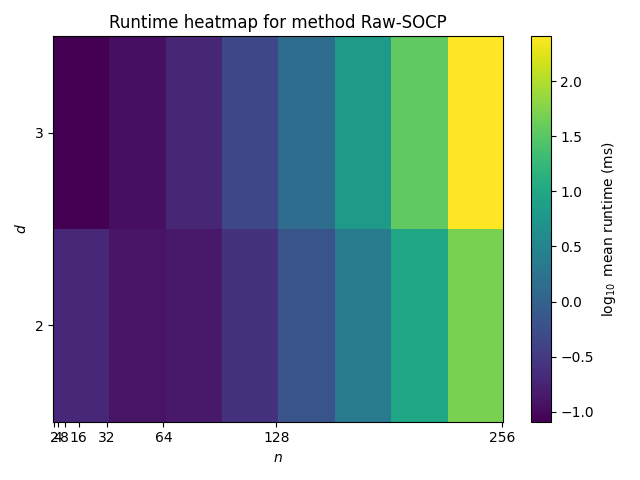}
\includegraphics[scale=.299]{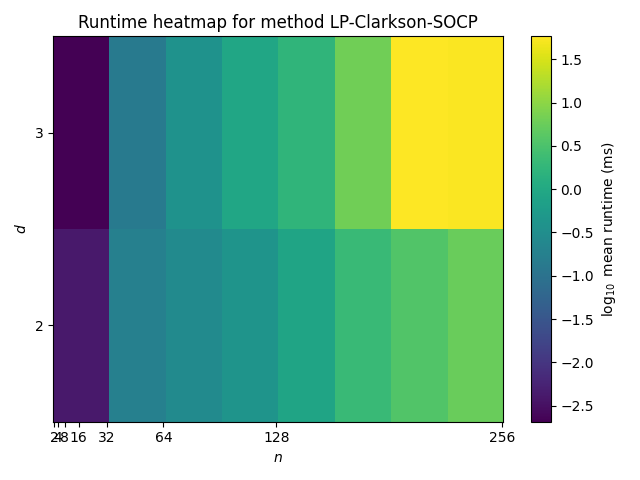}
\includegraphics[scale=.299]{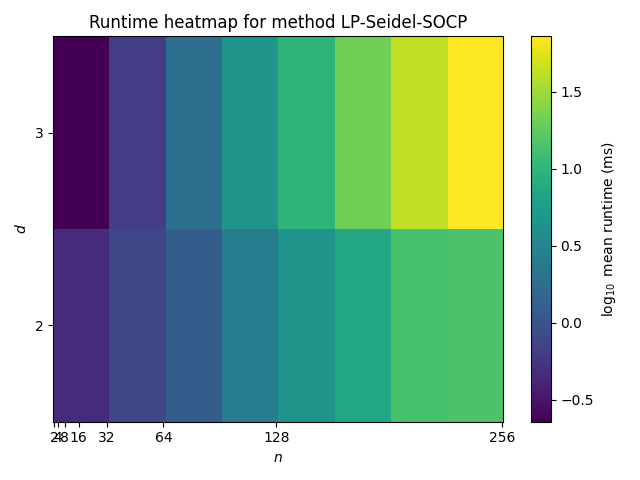}
\caption{Heatmaps of the mean runtime for the methods using SOCP, namely the raw solve (left), LP-Clarkson (center), and LP-Seidel (right), for $d=2,3$ and increasing $n$. In each plot, the horizontal axis is the number of ellipsoids, and the vertical axis is the ambient dimension.}
\label{fig_heatmaps_socp_d}
\end{figure}

\begin{figure}[htbp]
\centering
\includegraphics[scale=.299]{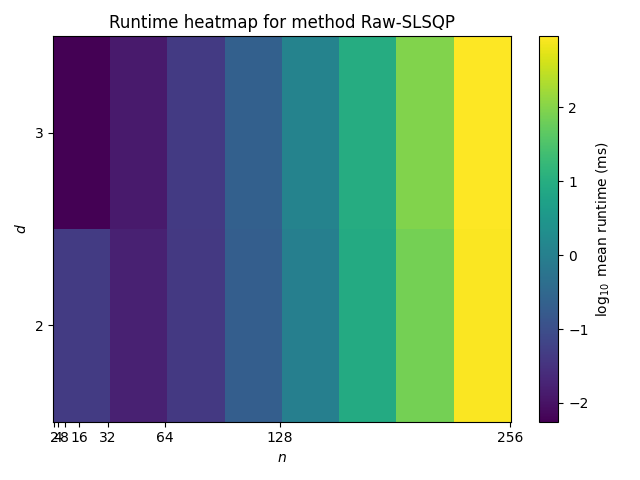}
\includegraphics[scale=.299]{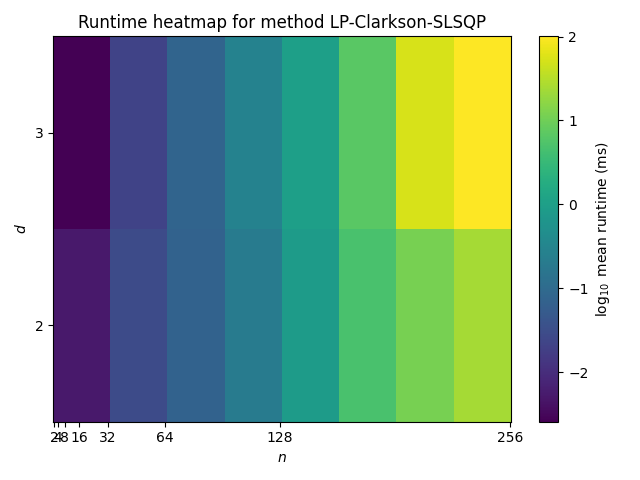}
\includegraphics[scale=.299]{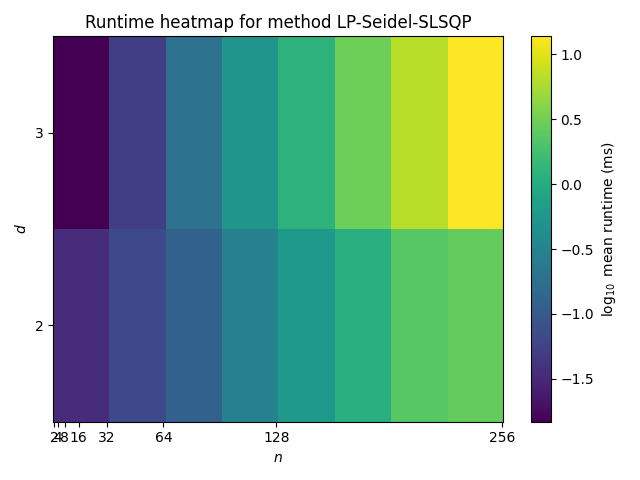}
\caption{Heatmaps of the mean runtime for the methods using SLSQP, namely the raw solve (left), LP-Clarkson (center), and LP-Seidel (right), for $d=2,3$ and increasing $n$. In each plot, the horizontal axis is the number of ellipsoids, and the vertical axis is the ambient dimension.}
\label{fig_heatmaps_slsqp_d}
\end{figure}

\begin{figure}[htbp]
\centering
\includegraphics[scale=.427]{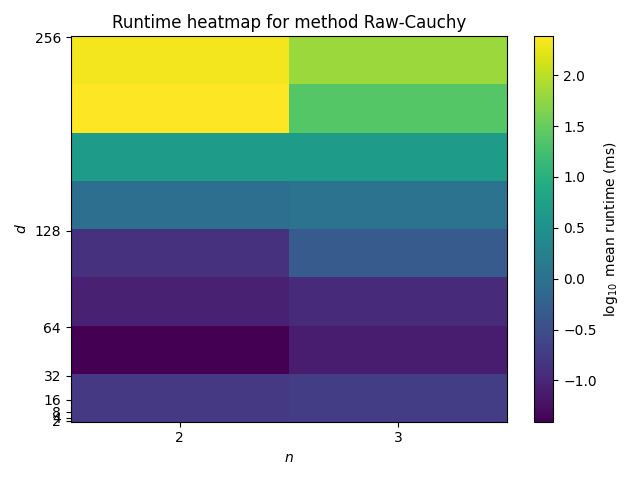}
\includegraphics[scale=.427]{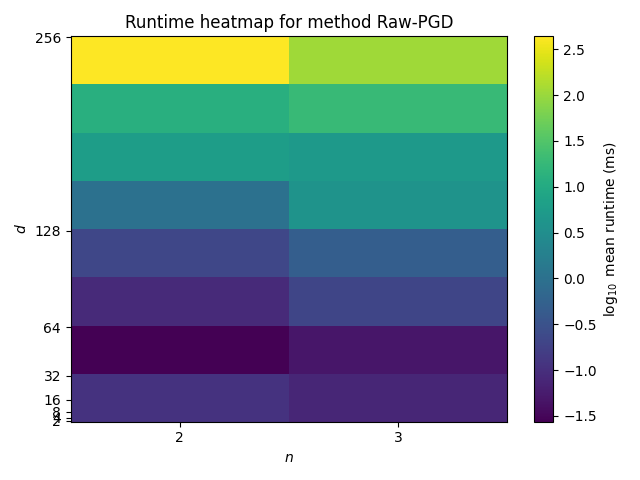}
\caption{Heatmaps of the mean runtime for the raw methods Cauchy (left) and PGD (right), for $n=2,3$ and increasing $d$. In each plot, the horizontal axis is the ambient dimension, and the vertical axis is the number of ellipsoids.}
\label{fig_heatmaps_cauchy_pgd_n}
\end{figure}

\begin{figure}[htbp]
\centering
\includegraphics[scale=.299]{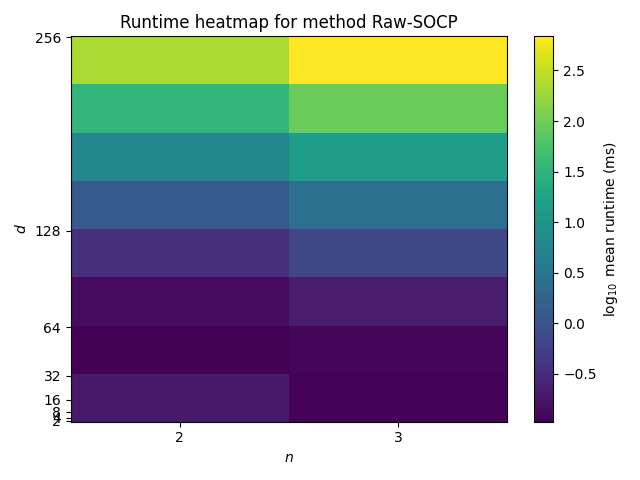}
\includegraphics[scale=.299]{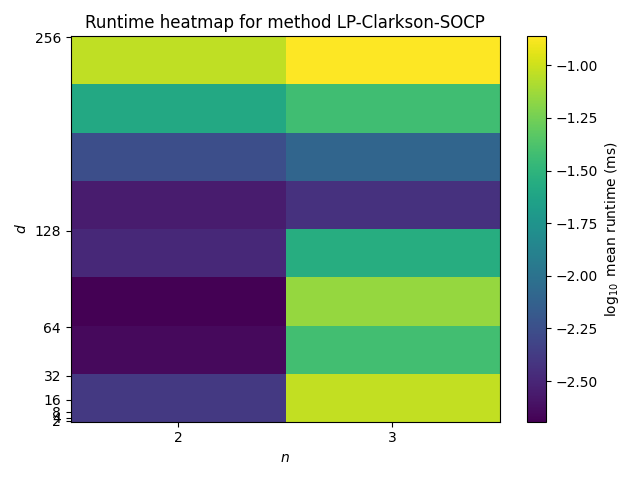}
\includegraphics[scale=.299]{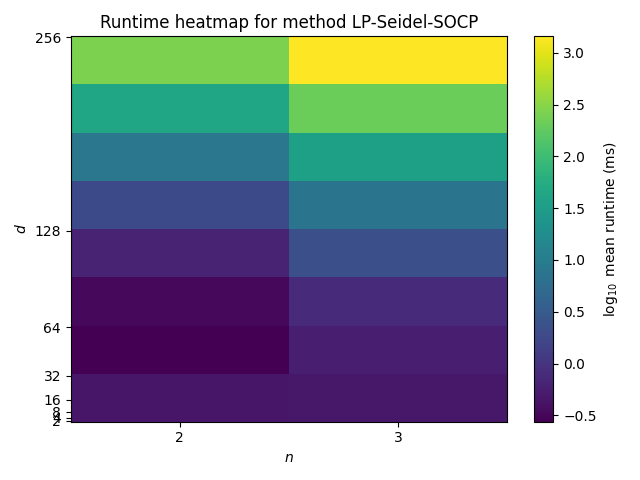}
\caption{Heatmaps of the mean runtime for the methods using SOCP, namely the raw solve (left), LP-Clarkson (center), and LP-Seidel (right), for $n=2,3$ and increasing $d$. In each plot, the horizontal axis is the ambient dimension, and the vertical axis is the number of ellipsoids.}
\label{fig_heatmaps_socp_n}
\end{figure}

\begin{figure}[htbp]
\centering
\includegraphics[scale=.299]{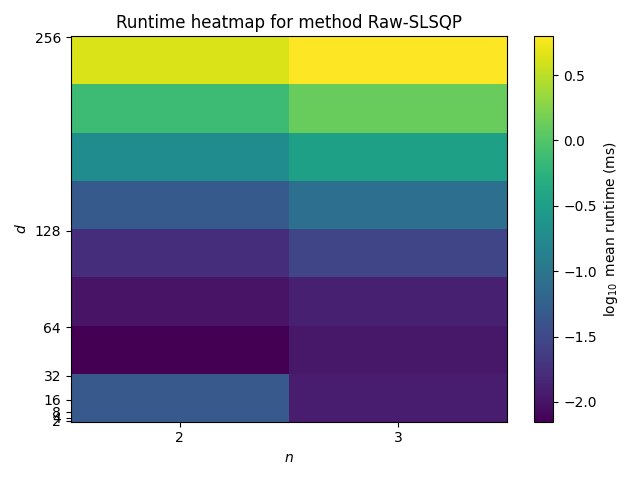}
\includegraphics[scale=.299]{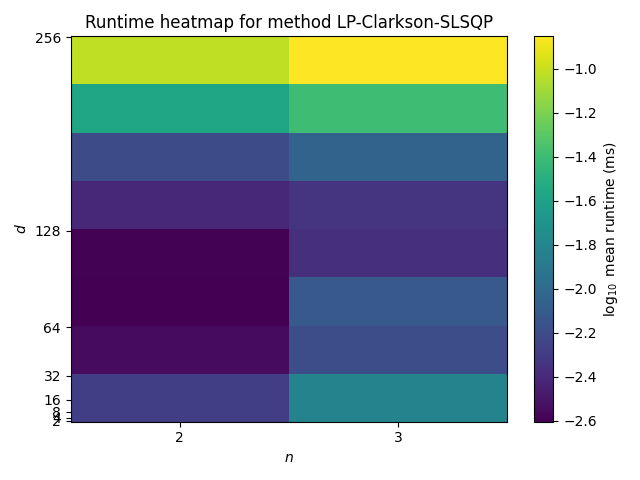}
\includegraphics[scale=.299]{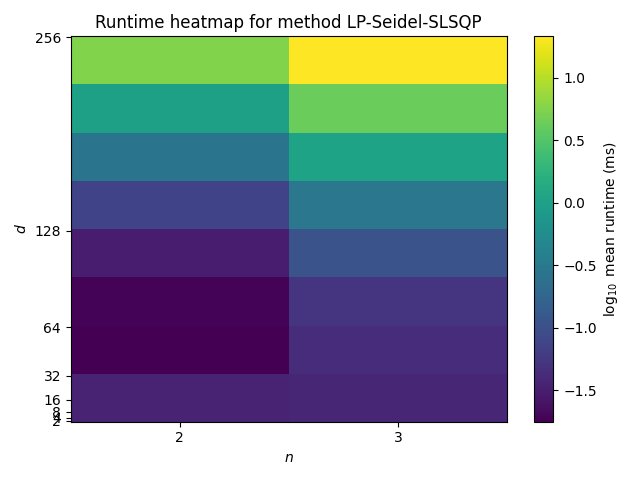}
\caption{Heatmaps of the mean runtime for the methods using SLSQP, namely the raw solve (left), LP-Clarkson (center), and LP-Seidel (right), for $n=2,3$ and increasing $d$. In each plot, the horizontal axis is the ambient dimension, and the vertical axis is the number of ellipsoids.}
\label{fig_heatmaps_slsqp_n}
\end{figure}

\begin{figure}[htbp]
\centering
\includegraphics[scale=.427]{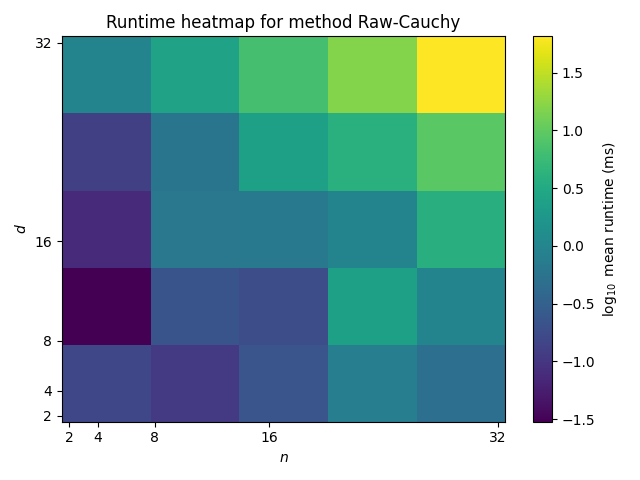}
\includegraphics[scale=.427]{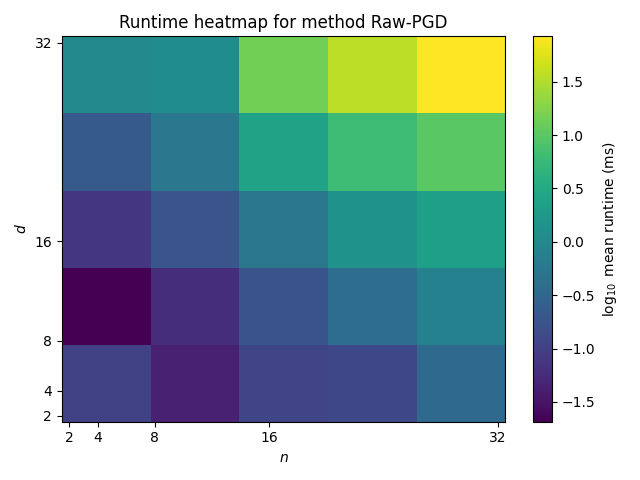}
\caption{Heatmaps of the mean runtime for the raw methods Cauchy (left) and PGD (right), for $n,d=2,4,8,16,32$. In each plot, the horizontal axis is the number of ellipsoids, and the vertical axis is the ambient dimension.}
\label{fig_heatmaps_cauchy_pgd_nd}
\end{figure}

\begin{figure}[htbp]
\centering
\includegraphics[scale=.299]{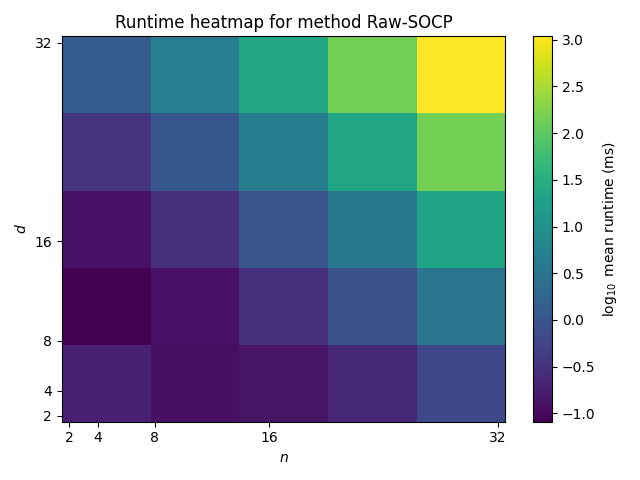}
\includegraphics[scale=.299]{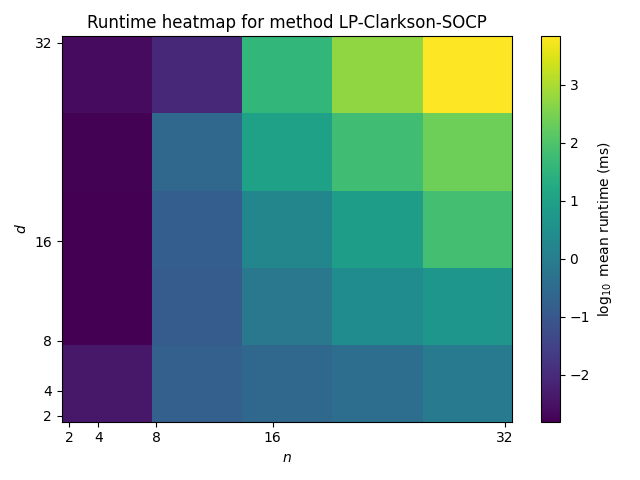}
\includegraphics[scale=.299]{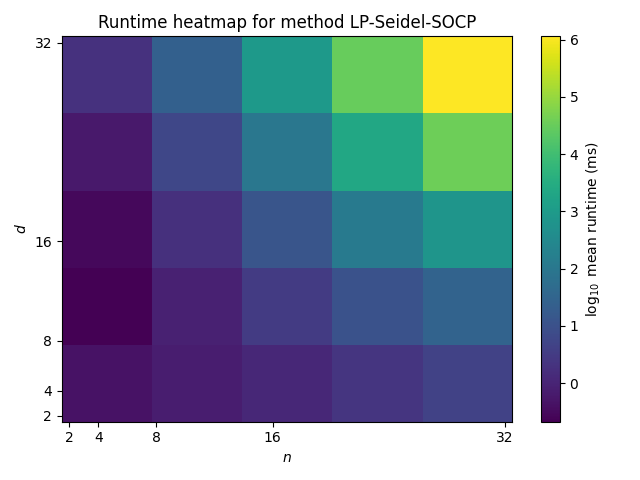}
\caption{Heatmaps of the mean runtime for the methods using SOCP, namely the raw solve (left), LP-Clarkson (center), and LP-Seidel (right), for $n,d=2,4,8,16,32$. In each plot, the horizontal axis is the number of ellipsoids, and the vertical axis is the ambient dimension.}
\label{fig_heatmaps_socp_nd}
\end{figure}

\begin{figure}[htbp]
\centering
\includegraphics[scale=.299]{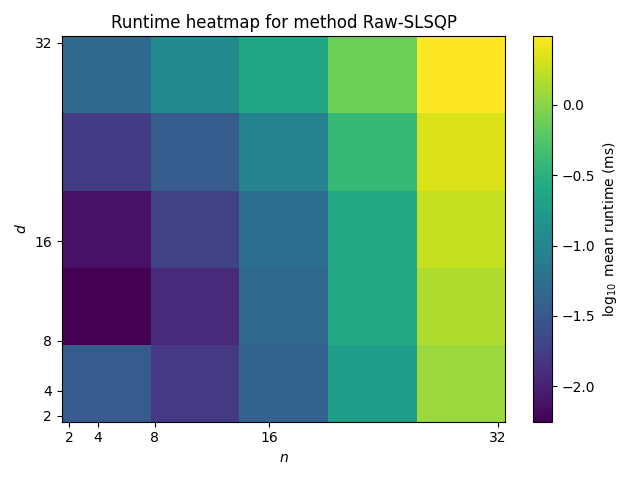}
\includegraphics[scale=.299]{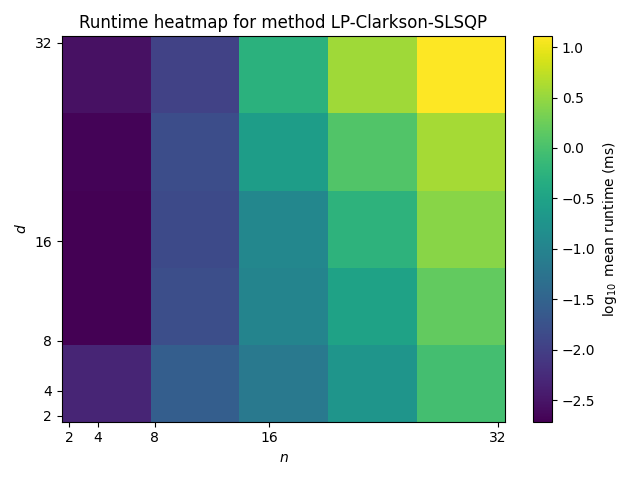}
\includegraphics[scale=.299]{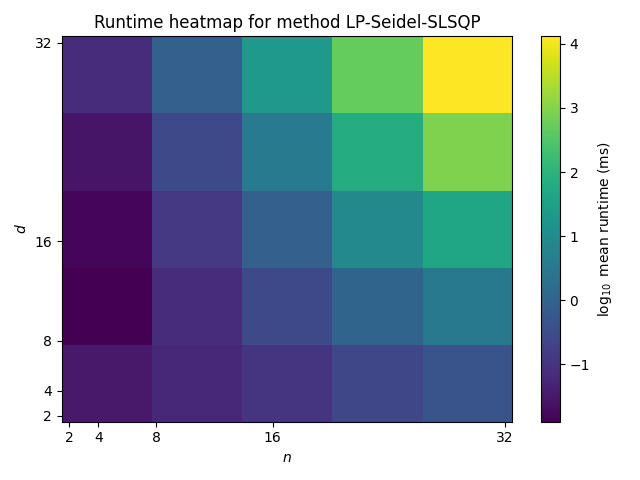}
\caption{Heatmaps of the mean runtime for the methods using SLSQP, namely the raw solve (left), LP-Clarkson (center), and LP-Seidel (right), for $n,d=2,4,8,16,32$. In each plot, the horizontal axis is the number of ellipsoids, and the vertical axis is the ambient dimension.}
\label{fig_heatmaps_slsqp_nd}
\end{figure}

\end{document}